\documentclass[10pt,reqno]{amsart}
\usepackage{amsmath} 
\usepackage{amssymb}
\usepackage{dsfont}
\usepackage[dvips,draft,final]{graphics}
\usepackage{amssymb,amsmath}
\usepackage[T1]{fontenc}
\usepackage{lmodern}
\usepackage{fancyhdr}
\usepackage{url}
\usepackage{mathrsfs}

\usepackage{color}
\usepackage{graphicx}

\newtheorem{theorem}{Theorem}[section]
\newtheorem{proposition}[theorem]{Proposition}
\newtheorem{lemma}[theorem]{Lemma}
\newtheorem{follow}[theorem]{Corollary}

\theoremstyle{definition}
\newtheorem{remark}[theorem]{Remark}

\newcommand{\bel}{\begin{equation} \label}
\newcommand{\ee}{\end{equation}}

\newcommand{\pd}{\partial}

\newcommand{\supp}{{\text{supp}}}

\newcommand{\C}{{\mathbb C}}
\newcommand{\R}{{\mathbb R}}

\newcommand{\re}{\mathfrak R}
\newcommand{\im}{\mathfrak I}

\newcommand{\sA}{{\mathscr A}}
\newcommand{\cB}{{\mathcal B}}
\newcommand{\cC}{{\mathcal C}}

\newcommand{\cH}{{\mathcal H}}
\newcommand{\sH}{{\mathscr H}}
\newcommand{\sK}{{\mathscr K}}
\newcommand{\sL}{{\mathscr L}}
\newcommand{\cM}{{\mathcal M}}

\newcommand{\cT}{{\mathcal T}}
\newcommand{\cU}{{\mathcal U}}

\newcommand{\sV}{{\mathscr V}}

\newcommand{\cZ}{{\mathcal Z}}
\newcommand{\Z}{{\mathbb Z}}
\def\epsilon{\varepsilon}
\def\phi {\varphi}

\def\beq{\begin{equation}}
\def\eeq{\end{equation}}
\renewcommand{\leq}{\leqslant}
\renewcommand{\geq}{\geqslant}
\newcommand{\bea}{\begin{eqnarray}}
\newcommand{\eea}{\end{eqnarray}}
\newcommand{\beas}{\begin{eqnarray*}}
\newcommand{\eeas}{\end{eqnarray*}}

{

\providecommand{\abs}[1]{\left\lvert#1\right\rvert}
\providecommand{\norm}[1]{\left\lVert#1\right\rVert}

\title[On the Calder\`on problem in periodic cylindrical domain with partial data]{On the Calder\`on problem in periodic cylindrical domain with partial Dirichlet and Neumann data}
\author{Mourad Choulli, Yavar Kian, Eric Soccorsi}

\date{}
\begin{document}
\begin{abstract}
We consider the Calder\`on problem in an infinite cylindrical domain, whose cross section is a bounded domain of the plane.
We prove log-log stability in the determination of the isotropic periodic conductivity coefficient from partial Dirichlet data and partial Neumann boundary observations of the solution.
\end{abstract}
\maketitle

%%%%%%%%%%%%%%%%%%%%%%%%%%%%%%%%%%%%%%%%%%%%%%%%%%%%%%%%%%%%%%%%%%%%%%%%%%%%%%%%%%%%%%%%%%%%%%%%%%%%%%%%%%%%%%%%%%%%%%%%%%%%%%%%%%%%%%%%%%%%%%%%%%%%%%%%%%%%%%%%%%%%%%%%%%%%%%%%%%%%%%%%

\section{Introduction}
\label{sec-intro}

%\subsection{What we are aiming for}
%\label{sec-aim}

Let $\omega$ be a bounded domain of $\mathbb{R}^2$ which contains the origin, with a $C^{2}$ boundary. 
%We assume without loss of generality that $\omega$ contains the origin. 
Set $\Omega:=\R \times \omega$ and denote any point $x \in \Omega$ by $x=(x_1,x')$, where $x_1 \in \R$ and $x':=(x_2,x_3) \in \omega$.
Given $V \in L^\infty(\Omega)$, real-valued and $1$-periodic with respect to $x_1$, i.e. 
\bel{a0}
V(x_1+1,x')=V(x_1,x'),\ x'\in\omega,\ x_1\in\R,
\ee
we consider the boundary value problem (BVP) with non-homogeneous Dirichlet data $f$,
\bel{eq1}
\left\{
\begin{array}{rcll} 
(-\Delta + V ) u& = & 0, & \mbox{in}\ \Omega,\\ 
u & = & f, & \mbox{on}\ \Gamma := \pd \Omega = \R \times \pd \omega.
\end{array}
\right.
\ee
%Let $\nu'$ be the outward unit normal to $\pd \omega$
Next we fix $\xi_0 \in \mathbb S^1:=\{ y \in \R^2;\ \abs{y} = 1 \}$ and define the $\xi_0$-shadowed (resp., $\xi_0$-illuminated) face of $\pd \omega$ as
\bel{a0b} 
\pd \omega_{\xi_0}^+ := \{ x' \in \pd \omega;\ \nu'(x') \cdot \xi_0 \geq 0\}\ (\mbox{resp.}\ \pd \omega_{\xi_0}^- := \{ x' \in \pd \omega;\ \nu'(x') \cdot \xi_0 \leq 0 \}),
\ee
where $\nu'$ is the outgoing unit normal vector to $\pd \omega$. Here and henceforth the symbol $\cdot$ denotes the Euclidian scalar product in $\R^k$, $k \geq 2$, and $\abs{y}:=(y \cdot y)^\frac{1}{2}$ for all $y \in \R^k$. 

Then for any closed neighborhood $G'$ of $\pd \omega_{\xi_0}^-$ in $\pd \omega$, we know from \cite[Theorem 1.1]{CKS2} that knowledge of the partial Dirichlet-to-Neumann (DN) map restricted to $G:=\R \times G'$, 
\bel{a1}
\Lambda_V : f \mapsto \pd_\nu u_{| G}, 
\ee
uniquely and logarithmic-stably determines $V$. Here we used the usual notation $\pd_\nu u := \nabla u \cdot \nu$,
where $\nabla$ denotes the gradient operator with respect to $x \in \Omega$, $u$ is the solution to \eqref{eq1}, and
$$ \nu(x_1,x') :=(0,\nu'(x')),\ x=(x_1,x') \in \Gamma, $$
is the outward unit vector normal to $\Gamma$. Otherwise stated, the unknown potential $V$ appearing in the first line of \eqref{eq1} can be stably recovered from boundary observation of the current flowing through $G$, upon probing the system \eqref{eq1} with non-homogeneous Dirichlet data. 

Notice that in the above mentioned result, only the output, i.e. the measurement of the current flowing across $\Gamma$, is local (or partial) in the sense that it is performed on $G$ and not on the whole boundary $\Gamma$, while the input, i.e. the Dirichlet data, remains global as it is possibly supported everywhere on $\Gamma$. Therefore, \cite[Theorem 1.1]{CKS2} claims logarithmic stability in the inverse problem of determining the electric potential $V$ in the first line of \eqref{eq1} from the knowledge of partial Neumann, and full Dirichlet, data. In the present paper, we aim for the same type of result under the additional constraint that not only the Neumann data, but also the Dirichlet data, be partial. Namely, given an arbitrary closed neighborhood $F'$ of $\pd \omega_{\xi}^+$ in $\pd \omega$, such that
\bel{a3}
F' \cap G' \neq \emptyset\ \mbox{and}\ F' \cup G' = \pd \omega,
\ee
we seek stable identification of $V$ by the input-restricted DN map \eqref{a1} to Dirichlet data functions $f$ supported in 
$F:= \R \times F'$.

%%%%%%%%%%%%%%%%%%%%%%%%%%%%%%%%%%%%%%%%%%%%%%%%%%%%%%%%%%%%%%%%%%%%%%%%%%

\subsection{State of the art}
Since the seminal paper \cite{Ca} by Calder\'on, the electrical impedance tomography problem, or Calder\'on problem, of retrieving the conductivity from the knowledge of the DN map on the boundary of a bounded domain, has attracted many attention. If the conductivity coefficient is scalar, then the Liouville transform allows us to rewrite the Calder\'on problem into the inverse problem of determining the electric potential in Laplace operator, from boundary measurements. There is an extensive literature on the Calder\'on problem.
For isotropic conductivities, a great deal of work has been spent to weaken the regularity assumption on the conductivity, in the study of the uniqueness issue, see e.g. \cite{BT,HT}. %In the two-dimensional case, the conductivity is needed to be bounded only, see \cite{B}). 
%In dimension $d \geq 3$, the conductivity must lie in the Sobolev space $W^{\frac{3}{2},p}$ with $p>2d$, see \cite{SU,BT}. 
%In dimension $d \geq 3$, this assumption conductivity must lie in the Sobolev space $W^{\frac{3}{2},p}$ with $p>2d$, see \cite{SU,BT}.
In all the above mentioned papers, the full DN map are needed, i.e. lateral observations are performed on the whole boundary. The first uniqueness result from partial data for the Calder\'on problem, was obtained in dimension 3 or greater, by Bukhgeim and Uhlmann in \cite{BU}. Their result, which requires that Dirichlet data be imposed on the whole boundary, and that Neumann data boundary be observed on slightly more than half of the boundary, was improved by Kenig, Sj\"ostrand and Uhlmann in \cite{KSU}, where both input and ouput data are measured on subsets of the boundary. In the two-dimensional case, Imanuvilov, Uhlmann and Yamamoto proved in \cite{IUY1, IUY2} that the partial DN map uniquely determines the conductivity. We also mention that the special case of the Calder\'on problem in a bounded cylindrical domain of $\R^3$, was treated in \cite{IY1}.

The stability issue for the Calder\'on problem was addressed by Alessandrini in \cite{Al}. He proved a log-type stability estimate with respect to the full DN map. Such a result, which is known to be optimal, see \cite{Ma}, degenerates to log-log stability with partial Neumann data, see \cite{HW, CKS1}. In \cite{CDR1, CDR2}, Caro, Dos Santos Ferreira and Ruiz proved stability results of log-log type, corresponding to the uniqueness results of \cite{KSU} in dimension 3 or greater. We refer to \cite{BIY,NSa,Sa} for stability estimates associated with the two-dimensional Calder\'on problem, and we point out that both the electric and the magnetic potentials are stably determined by the partial DN map in \cite{T}.

Notice that all the above mentioned papers are concerned with the Calder\'on problem in a bounded domain. It turns out that there is only a small number of mathematical papers dealing with inverse coefficients problems in an unbounded domain. Several authors considered the problem of recovering  coefficients in an unbounded domain from boundary measurements. Object identification in an infinite slab, was proved in \cite{Ik, SW}. Unique determination of a compactly supported electric potential of the Laplace equation in an infinite slab, by partial DN map, is established in \cite{LU}. This result was extended to the magnetic case in \cite{KLU}, and to bi-harmonic operators in \cite{Y}. 

The stability issue in inverse coefficients problems stated in an infinite cylindrical waveguide is addressed in \cite{BKS, CKS, CKS1,CS,KKS,Ki,KPS1,KPS2}, and a log-log type stability estimate by partial Neumann data for the periodic electric potential of the Laplace equation can be found in \cite{CKS2}. In the present paper, we are aiming for the same result as in \cite{CKS2}, where the full Dirichlet data is replaced by partial voltage. The proof the corresponding stability estimate relies on two different types of complex geometric optics (CGO) solutions to the quasi-periodic Laplace equation in $(0,1) \times \omega$, which are supported in $F$. These functions are built in Section \ref{sec-proofpr1} by means of a suitable Carleman estimate. This technique is inspired by \cite{KSU}, but, in contrast to \cite{CDR1,CDR2}, due to the quasi-periodic boundary conditions imposed on the CGO solution, we cannot apply the Carleman estimate of \cite{KSU}, here.
%\subsection{Notations}
%In this subsection we collect basic notations used throughout the entire text.

%Since the outward unit normal vector $\nu$ to $\Gamma$ reads
%$$ \nu(x_1,x')=(0,\nu'(x')),\ x=(x_1,x') \in \Gamma, $$
%where $\nu'$ is the unit outgoing normal vector to $\pd \omega$, then for notational simplicity, we refer to $\nu$ as both the %outward unit normal vector to $\Gamma$ and to $\pd \omega$. Similarly, we denote by $\cdot$ the scalar product in $\R^k$, for $k %\geq 2$, and put $\abs{y}:=( y \cdot y )^\frac{1}{2}$ for any $y \in \R^k$.

%%%%%%%%%%%%%%%%%%%%%%%%%%%%%%%%%%%%%%%%%%%%%%%%%%%%%%%%%%%%%%%%%%%%%%%%%%
\subsection{Settings and main result}
\label{sec-set}
We stick with the notations of \cite{CKS2} and denote by $C_\omega$ the square root of the first eigenvalue of the Dirichlet Laplacian in $L^2(\omega)$, that is the largest of those positive constants $c$, such that the Poincar\'e inequality 
\bel{inegp1}
\| \nabla' u \|_{L^2(\omega)} \geq c \| u \|_{L^2(\omega)},\ u \in H_0^1(\omega),
\ee
holds true. Here $\nabla':=(\pd_{x_2},\pd_{x_3})$ stands for the gradient with respect to $x'=(x_2,x_3)$. This can be equivalently reformulated as
\bel{p1const}
C_\omega := \sup \{ c >0\ \mbox{satisfying}\ \eqref{inegp1} \}.
\ee
Next, for $M_- \in (0,C_\omega)$ and $M_+ \in [M_-,+\infty)$, we introduce the set $\mathscr{V}_\omega(M_\pm)$ of admissible unknown potentials in the same way as in \cite[Sect. 1.2]{CKS2}:
\bel{admpot}
\mathscr{V}_\omega(M_\pm) := \{ V \in L^{\infty}(\Omega;\R)\ \mbox{satisfying}\ \eqref{a0},\ \| V \|_{L^{\infty}(\Omega)} \leq M_+\ \mbox{and}\ \| \max(0,-V) \|_{L^{\infty}(\Omega)} \leq M_- \}.
\ee

%Let $Y'$ be either $\omega$ or $\pd \omega$.
%For $r$ and $s$ in $\R$, we denote by $\cH^{r,s}(\R \times Y')$ the set $H^r(\R;H^s(Y'))$. Evidently we write $\cH^{r,s}(\Omega)$ %(resp., $\cH^{r,s}(\Gamma)$, $\cH^{r,s}(G)$) instead of
%$\cH^{r,s}(\R \times \omega)$ (resp., $\cH^{r,s}(\R \times \pd \omega)$). Although this notation is reminiscent of the one used by %Lions and Magenes in \cite{LM1} for anisotropic Sobolev spaces $H^r(\R;L^2(Y)) \cap L^2(\R;H^s(Y))$, it is worth noticing that they %do not coincide with $\cH^{r,s}(\R \times Y)$ unless we have $r=s=0$. 
%Next, we see for each $r>0$ and $s>0$ that $\cH^{-r,-s}(\R \times Y)$ is canonically identified with the space dual to %$\cH_0^{r,s}(\R \times Y)$, with respect to the pivot space
%$\cH^{0,0}(\R \times Y)=L^2(\R \times Y)$. Hee we have set $\cH_0^{r,s}(\R \times Y):=H^r(\R;H_0^s(Y))$, where $H_0^s(Y)$ denotes %the closure of $C_0^{\infty}(Y)$ in the topology of the Sobolev space $H^s(Y)$. 

%Let $X_1$ and $X_2$ be two Hilbert spaces. We denote by $\B(X_1,X_2)$ the class of bounded operators $T : X_1 \to X_2$. If %$X_1=X_2=X$ we write $\B(X)$ instead of $\B(X,X)$.
%Prior to stating the main result of this article we introduce several functional spaces needed for the definition of the DN map %$\Lambda_V$.
Before stating the main result of this paper we need to define the DN map associated with the BVP \eqref{eq1} and $V \in \mathscr{V}_\omega(M_\pm)$. 
To this end, we introduce the Hilbert space $H_\Delta(\Omega):=\{ u \in L^2(\Omega);\ \Delta u \in L^2(\Omega)\}$, endowed with the norm
$$ \norm{u}^2_{H_\Delta(\Omega)} :=\norm{u}_{L^2(\Omega)}^2+\norm{\Delta u}_{L^2(\Omega)}^2, $$
and refer to \cite[Lemma 2.2]{CKS2} in order to extend the mapping 
$$ \cT_0 u := u_{\vert \Gamma}\ (\mbox{resp.}, \cT_1 u := \pd_\nu u _{\vert \Gamma}),\ u \in C_0^\infty(\overline{\Omega}), $$
into a continuous function $\mathcal T_0 : H_\Delta(\Omega) \to H^{-2}(\R ; H^{-\frac{1}{2}}(\pd \omega))$ (resp., $\mathcal T_1 : H_\Delta(\Omega) \to H^{-2}(\R ; H^{-\frac{3}{2}}(\pd \omega))$). 
Since $\cT_0$ is one-to-one from $B:=\{ u \in L^2(\Omega);\ \Delta u = 0 \}$ onto
$$ \sH(\Gamma):= \mathcal T_0 H_\Delta(\Omega) = \{ \mathcal T_0 u;\ u \in H_\Delta(\Omega) \},  $$
by \cite[Lemma 2.3]{CKS2}, we put
\bel{nhg}
\norm{f}_{\sH(\Gamma)} : =\norm{\mathcal T_0^{-1} f}_{H_\Delta(\Omega)} = \norm{\mathcal T_0^{-1} f}_{L^2(\Omega)},
\ee
where $\mathcal T_0^{-1}$ denotes the operator inverse to $\mathcal{T}_0 : B \to \sH(\Gamma)$. Throughout this text, we consider Dirichlet data in $\sH(\Gamma)$ which are supported in $F$, i.e. input functions belonging to
$$ \sH_c(F):= \{ f \in \sH(\Gamma);\ \supp f \subset F \}. $$
To any $f \in \sH_c(F)$, we associate the unique solution $u \in H_\Delta(\Omega)$ to \eqref{eq1}, given by \cite[Proposition 1.1 (i)]{CKS2}, and define the partial DN map associated with \eqref{eq1}, as
\bel{es0}
\Lambda_V : f \in \sH_c(F) \mapsto \cT_1 u_{| G}.
\ee
Upon denoting by $\cB(X_1,X_2)$, where $X_j$, $j=1,2$, are two arbitrary Banach spaces, the class of bounded operators $T : X_1 \to X_2$, we recall from \cite[Proposition 1.1 (ii)-(iii)]{CKS2} that
\bel{es1}
 \Lambda_V \in \cB(\sH_c(F), H^{-2}(\R,H^{- \frac{3}{2}}(G')))\ \mbox{and}\
\Lambda_V-\Lambda_W \in \cB( \sH_c(F)  , L^2(G)),\ V,\ W \in
\mathscr{V}_\omega(M_\pm). 
\ee

The main result of this article, which claims that unknown potentials of $\mathscr{V}_\omega(M_\pm)$ are stably determined in the elementary cell $\check{\Omega}:=(0,1)\times\omega$, by the partial DN map, is stated as follows.

\begin{theorem}
\label{thm1} 
Let $V_j \in \mathscr{V}_\omega(M_\pm)$, $j=1,2$, where $M_+ \in [M_-,+\infty)$, $M_- \in (0,C_\omega)$, and $C_\omega$ is defined by \eqref{p1const}.
Then, there exist two constants  $C>0$ and $\gamma_*>0$, both of them depending only on $\omega$, $M_\pm$, $F'$, and $G'$, such that the estimate
\bel{thm1a} 
\norm{V_1-V_2}_{H^{-1}(\check{\Omega})} \leq C \Phi \left(\norm{\Lambda_{V_1}-\Lambda_{V_2}}\right),
\ee
holds with
\bel{es2}
\Phi(\gamma) := \left\{ \begin{array}{cl} \gamma & \mbox{if}\ \gamma \geq \gamma^*,\\
(\ln \abs{\ln \gamma})^{-1} & \mbox{if}\ \gamma \in (0, \gamma^*),\\ 
0 & \mbox{if}\ \gamma=0.
\end{array} \right.
\ee
Here $\| \cdot \|$ denotes the usual norm in $\cB(\sH_{c}(F),L^2(G))$. 
\end{theorem}

The statement of Theorem \ref{thm1} remains valid for any periodic potential $V \in L^{\infty}(\Omega)$, provided $0$ lies in the resolvent set of $A_V$, the self-adjoint realization  in $L^2(\Omega)$ of the Dirichlet Laplacian $-\Delta + V$. In this case, the multiplicative constants $C$ and $\gamma_*$, appearing in \eqref{thm1a}-\eqref{es2}, depend on (the inverse of) the distance $d>0$, between $0$ and the spectrum of $A_V$. In the particular case where $V \in \sV_\omega(M_\pm)$, with $M_- \in (0,C_\omega)$, we have $d \geq C_\omega- M_-$, and the implicit condition $d>0$ imposed on $V$, can be replaced by the explicit one on the negative part of the potential, i.e. $\| \max(0,-V) \|_{L^\infty(\Omega)} \leq M_-$. 
%The same remark applies for the statements of Corollary \ref{cor-a} and Theorem \ref{thm2}, below.

%%%%%%%%%%%%%%%%%%%%%%%%%%%%%%%%%%%%%%%%%%%%%%%%%%%%%%%%%%%%%%%%%%%%%%%%%%
\subsection{Application to the Calder\'on Problem}
\label{sec-CalderonPb}
The inverse problem addressed in Subsection \ref{sec-set} is closely related to the periodic Calder\'on problem in $\Omega$, i.e. the inverse problem of determining the conductivity coefficient $a$, obeying
\bel{a-per}
a(x_1+1,x') = a(x_1,x'),\ x'\in\omega,\ x_1 \in \R,
\ee
from partial boundary data of the BVP in the divergence form
\bel{a-eq1}
\left\{
\begin{array}{rcll} 
-\mbox{div}( a \nabla u ) & = & 0, & \mbox{in}\ \Omega,\\ 
u & = & f, & \mbox{on}\ \Gamma.
\end{array}
\right.
\ee
Let $\cT_0$ denote the trace operator $u \mapsto u_{\vert \Gamma}$ on $H^1(\Omega)$. We equip the space
$\sK(\Gamma):=\cT_0(H^1(\Omega))$ with the norm
$$ \norm{f}_{\sK(\Gamma)}:=\inf\{\norm{u}_{H^1(\Omega)};\ \cT_0 u = f\}, $$
and recall for any $a \in C^1(\overline{\Omega})$ satisfying the ellipticity condition
\bel{ca1}
a(x) \geq a_* >0,\ x \in \Omega,
\ee
for some fixed positive constant $a_*$, that the BVP \eqref{a-eq1} admits a unique solution $u \in H^1(\Omega)$ for each $f \in \sK(\Gamma)$.
Moreover, the full DN map associated with \eqref{a-eq1}, defined by $f \mapsto a \cT_1 u$, where $\cT_1 u := \partial_\nu u_{\vert \Gamma}$, 
is a bounded operator from $\sK(\Gamma)$ to $H^{-1}(\R;H^{-\frac{1}{2}}(\pd \omega))$. Here, we rather consider the partial DN map,
\bel{ca2} 
\Sigma_a : f \in \sK(\Gamma) \cap a^{-\frac{1}{2}}(\sH_c(F)) \mapsto a \cT_1 u_{\vert G},
\ee
where $a^{-\frac{1}{2}}(\sH_c(F)):= \{a^{-\frac{1}{2}}f;\ f \in \sH_c(F) \}$.

Further, since the BVP \eqref{a-eq1} is brought by the Liouville transform into the form \eqref{eq1}, with $V_a:=a^{-\frac{1}{2}} \Delta a^{\frac{1}{2}}$, then, 
with reference to \eqref{admpot}, we impose that $V_a$ be bounded in $\Omega$ and satisfies the following conditions
\bel{ca2b}
\| V_a \|_{L^{\infty}(\Omega)} \leq M_+\ \mbox{and}\ \| \max(0,-V_a) \|_{L^{\infty}(\Omega)} \leq M_-,
\ee
where $M_- \in (0,C_\omega)$ and $M_+ \in [M_-,+\infty)$ are {\it a priori} arbitrarily fixed constants.
Namely, we introduce the set of admissible conductivities, as
\bel{def-a} 
\sA_\omega(a_*,M_\pm) := \left\{ a \in C^1(\overline{\Omega};\R)\ \mbox{satisfying}\ \Delta a \in L^{\infty}(\Omega),\ \| a \|_{W^{1,\infty}(\Omega)} \leq M_+, \eqref{a-per}, \eqref{ca1}, \mbox{and}\ \eqref{ca2b} \right\}.
\ee
We check by standard computations that the condition \eqref{ca2b}
%$$\| V_a \|_{L^{\infty}(\Omega)} \leq M_+\ \mbox{and}\ \| \max(0,-V_a) \|_{L^{\infty}(\Omega)} \leq M_-,$$ 
is automatically verified, provided the conductivity $a \in \sA_\omega(a_*,M_\pm)$ is taken so small that $\| a \|_{W^{1,\infty}(\Omega)}^2 + 2 a_* \| \Delta a \|_{L^{\infty}(\Omega)} \leq 4 M_- a_*^2$, or even that
$$ \| a \|_{W^{2,\infty}(\Omega)} \leq \frac{4 M_-}{(4 M_- +1)^{\frac{1}{2}}+1} a_*, $$
in the particular case where $a \in W^{2,\infty}(\Omega)$.
%Therefore, for convenience, we shall systematically assume (upon possibly enlarging $M_+$) in the foregoing, that we have
%$$ \left( a \in \sA_\omega(a_*,M_\pm) \right) \Longrightarrow \left( \| a \|_{W^{2,\infty}(\check{\Omega})} \leq M_+ \right). $$

%Here, we associate to $a^{-\frac{1}{2}}\mathcal H(\Gamma)$ the norm
%\[\norm{f}_{a_1^{-\frac{1}{2}}\mathcal H(\Gamma)}=\norm{a_1^{\frac{1}{2}}f}_{\mathcal H(\Gamma)}.\]
The main result of this section claims stable determination of such admissible conductivities $a$, from the knowledge of $\Sigma_a$. It is stated as follows.

\begin{follow}
\label{cor-a} 
Fix $a_*>0$, and let $M_\pm$ be as in Theorem \ref{thm1}. Pick $a_j \in \sA_\omega(a_*,M_\pm)$, for $j=1,2$, obeying
\bel{ca3} 
a_1(x)=a_2(x),\ x \in \pd \Omega\ %\mbox{and}\ \pd_\nu a_1(x)=\pd_\nu a_2(x),\ x \in F \cap G.
\ee
and
\bel{ca4} 
\pd_\nu a_1(x)=\pd_\nu a_2(x),\ x \in F \cap G.
\ee
Then $\Sigma_{a_1}-\Sigma_{a_2}$ is extendable to a bounded operator from $a_1^{-\frac{1}{2}}(\sH_c(F))$ into $L^2(G)$. Moreover, there exists two constant $C>0$ and $\gamma_*>0$, both of them depending only on $\omega$, $M_\pm$, $a_*$, $F'$, and $G'$, such that we have
\bel{ca5}
\norm{a_1-a_2}_{H^1(\check{\Omega})} \leq C \Phi \left( a_*^{-\frac{1}{2}} \norm{\Sigma_{a_1}-\Sigma_{a_2}} \right),
\ee
where $\Phi$ is the same as in Theorem \ref{thm1}.
Here $\norm{\cdot}$ denotes the usual operator norm in $\cB(a_1^{-\frac{1}{2}}(\sH_c(F)), L^2(G))$.
\end{follow}

\subsection{Floquet decomposition}
\label{sec-Floquet}
In this subsection, we reformulate the inverse problem presented in Subsection \ref{sec-set} into a family of inverse coefficients problems associated with the BVP
\bel{eq2}
\left\{
\begin{array}{rcll} 
(-\Delta + V ) v & = & 0, & \mbox{in}\ \check{\Omega}:=(0,1) \times \omega,\\ 
v & = & g, & \mbox{on}\ \check{\Gamma} := (0,1) \times \pd \omega, \\
v(1,\cdot) - e^{i \theta} v(0,\cdot) & = & 0, & \mbox{in}\ \omega, \\
\pd_{x_1} v(1,\cdot) - e^{i \theta} \pd_{x_1} v(0,\cdot) & = & 0, & \mbox{in}\ \omega,
\end{array}
\right.
\ee
for $\theta \in [0, 2 \pi)$, and suitable Dirichlet data $g$.
This is by means of the Floquet-Bloch-Gel'fand (FBG) transform introduced in \cite[Section 3.1]{CKS2}. We stick with the notations of \cite[Section 3.1]{CKS2}, and, for $Y$ being either $\omega$ of $\pd \omega$, we denote by $\cU$ the FBG transform from $L^2(\R \times Y)$ onto $\int_{(0,2\pi)}^\oplus L^2((0,1) \times Y) \frac{d \theta}{2 \pi}$. That is to say, the FBG transform $\cU$ maps $L^2(\Omega)$ onto $\int_{(0,2\pi)}^\oplus L^2(\check{\Omega}) \frac{d \theta}{2 \pi}$ if $Y=\omega$, and $L^2(\Gamma)$ onto $\int_{(0,2\pi)}^\oplus L^2(\check{\Gamma}) \frac{d \theta}{2 \pi}$ when $Y=\pd \omega$. We recall that the operator $\cU$ is unitary in both cases. We start by introducing several functional spaces and trace operators that are needed by the analysis of the inverse problem associated with \eqref{eq2}.

\subsubsection{Functional spaces and trace operators}
Fix $\theta \in [0,2 \pi)$.
With reference to \cite[Section 6.1]{CKS} or \cite[Section 3.1]{CKS2}, we set for each $n \in \mathbb N \cup \{ \infty \}$, 
$$ \mathcal C_{ \theta}^n \left([0,1] \times \overline{\omega}\right):= \left\{ u \in \mathcal C^n\left( [0,1]\times\overline{\omega}\right);\ \pd_{x_1}^j u(1,\cdot) -e^{i\theta} \pd_{x_1}^j u(0,\cdot) = 0\ \mbox{in}\ \omega,\ j \leq n \right\}, $$
and for $Y$ being either $\omega$ or $\pd \omega$, we put
$$ H_\theta^s((0,1) \times Y):= \left\{ u \in H^s((0,1)\times Y);\ \pd_{x_1}^j u(1,\cdot) - e^{i \theta} \pd_{x_1}^j u(0,\cdot) = 0\ \mbox{in}\ \omega,\ j < s - \frac{1}{2} \right\}\ \mbox{if}\ s > \frac{1}{2}, $$
and
$$ H_\theta^s((0,1)\times Y):= H^s((0,1) \times Y)\ \mbox{if}\ s \in \left[ 0 , \frac{1}{2} \right].  $$

Further, we recall from \cite[Eq. (3.29)]{CKS2} that $\cU H_\Delta(\Omega) = \int_{(0,2\pi)}^\oplus H_{\Delta,\theta}(\check{\Gamma}) \frac{d \theta}{2 \pi}$, where
$$ H_{\Delta,\theta}(\check{\Omega}) :=\{ u\in L^2(\check{\Omega});\ \Delta u \in L^2(\check{\Omega})\ \mbox{and}\ u(1,\cdot) - e^{i\theta} u(0,\cdot)= \pd_{x_1} u(1,\cdot) - e^{i\theta} \pd_{x_1} u(0,\cdot) = 0\ \mbox{in}\ \omega \}.
$$
Moreover, the space $\cC_\theta^\infty\left([0,2\pi]\times\overline{\omega}\right)$ is dense in $H_{\Delta,\theta}(\check{\Omega})$, and we have $\cU \cT_j \cU^{-1} = \int_{(0,2\pi)}^\oplus \cT_{j,\theta} \frac{d \theta}{2 \pi}$ for $j=0,1$, where the linear bounded operator
$$ \mathcal T_{j,\theta} : H_{\Delta,\theta}(\check{\Omega}) \to H^{-2}_\theta(0,1, H^{-\frac{2j+1}{2}}(\pd \omega)),$$ 
fulfills $\cT_{0,\theta} u = u_{\vert \check{\Gamma}}$ if $j=0$, and $\cT_{1,\theta} u=\pd_\nu u_{\vert \check{\Gamma}}$ if $j=1$, provided $u \in \cC_\theta^\infty \left([0,1] \times \overline{\omega}\right)$.
Therefore, putting
$$ \sH_\theta(\check{\Gamma}) :=\{ \cT_{0,\theta}u;\ u \in H_{\Delta,\theta}(\check{\Omega}) \},\ \mbox{and}\ \sH_{c,\theta}(\check{F}) := \{ f \in \sH_\theta(\check{\Gamma}),\ \mbox{supp}\ f \subset \check{F} \}, $$
we get that $\cU \sH(\Gamma) = \int_{(0,2\pi)}^\oplus \sH_\theta(\check{\Gamma}) \frac{d \theta}{2 \pi}$ and $\cU \sH_c(F) = \int_{(0,2\pi)}^\oplus \sH_{c,\theta}(\check{F}) \frac{d \theta}{2 \pi}$. As in \cite[Eq. (3.30)]{CKS2}, the space $\sH_{\theta}(\check{\Gamma})$ is endowed with the norm $\norm{g}_{\sH_\theta(\check{\Gamma})} := \norm{v_g}_{L^2(\check{\Omega})}$, where $v_g$ denotes the unique $L^2(\check{\Omega})$-solution to \eqref{eq2} with $V=0$, given by \cite[Proposition 3.2 (i)]{CKS2}.

\subsubsection{Inverse fibered problems}
Let $V \in \sV_\omega(M_\pm)$, where $M_\pm$ are as in Theorem \ref{thm1}.
Then, for any $f$ be in $\sH_c(F)$, $u$ is the $H_{\Delta}(\Omega)$-solution to \eqref{eq1}, if and only if, for almost every $\theta \in [0, 2 \pi)$,  $( \cU u )_\theta$ is the $H_{\Delta}(\check{\Omega})$-solution to \eqref{eq2}, associated with $g=( \cU f )_\theta \in \sH_{c,\theta}(\check{F})$.
The corresponding partial DN map, defined by $\Lambda_{V,\theta} : g \in \sH_{c,\theta}(\check{F}) \mapsto \cT_{1,\theta} v_{| \check{G}}$, where $v$ is the unique $H_{\Delta}(\check{\Omega})$-solution to \eqref{eq2}, is a bounded operator from $\sH_{c,\theta}(\check{F})$ into $H_\theta^{-2}(0,1;H^{-\frac{3}{2}}(G'))$, and we have
\bel{d3} 
\cU \Lambda_V \cU^{-1} = \int_{(0,2\pi)}^\oplus \Lambda_{V,\theta} \frac{d \theta}{2 \pi},
\ee
according to \cite[Proposition 7.1]{CKS2}. Further, if $V_1$ and $V_2$ are two potentials lying in $\sV_\omega(M_\pm)$, then $\Lambda_{V_1,\theta} - \Lambda_{V_2,\theta} \in \cB(\sH_{c,\theta}(\check{F}), L^2(\check{G}))$, for each $\theta \in [0, 2 \pi)$, by
\eqref{es1} and \eqref{d3}. Moreover, $\Lambda_{V_1} - \Lambda_{V_2}$ being unitarily equivalent to the family of partial DN maps $\{ \Lambda_{V_1,\theta} - \Lambda_{V_2,\theta},\ \theta \in [0, 2 \pi) \}$ , it holds true that
\bel{d4}
\| \Lambda_{V_1}-\Lambda_{V_2}  \|_{\cB(\sH_c(F) , L^2(G) )} = \sup_{\theta \in [0,2 \pi)} \| \Lambda_{V_1,\theta}-\Lambda_{V_2,\theta} \|_{\cB(\sH_{c,\theta}(\check{F}),L^2(\check{G}))}.
\ee
Therefore, it is clear from \eqref{d4} that Theorem \ref{thm1} is a byproduct of the following statement.

\begin{theorem}
\label{thm2} 
Let $M_\pm$ and $V_j$, $j=1,2$, be as in Theorem \ref{thm1}. Fix $\theta \in [0,2\pi)$. Then, there exist two constants $C_\theta>0$ and $\gamma_{\theta,*}>0$, both of them depending only on $\omega$, $M_\pm$, $F'$, and $G'$, such that we have
\bel{thm2a} 
\norm{V_1-V_2}_{H^{-1}(\check{\Omega})} \leq C_\theta \Phi_\theta \left( \norm{\Lambda_{{V}_1,\theta}-\Lambda_{{V}_2,\theta}} \right).
\ee
Here, $\Phi_\theta$ is the function defined in Theorem \ref{thm1}, upon substituting $\gamma_{\theta,*}$ for $\gamma$ in \eqref{es2}, and $\norm{\cdot}$ denotes the usual norm in $\cB(\sH_{c,\theta}(\check{F}),L^2(\check{G}))$.
\end{theorem}

We notice that the constants $C_{\theta}$ and $\gamma_{\theta,*}$ of Theorem \ref{thm2}, may possibly depend on $\theta$. Nevertheless, we infer from \eqref{d4} that this is no longer the case for $C$ and $\gamma$, appearing in the stability estimate \eqref{thm1a} of Theorem \ref{thm1}, as we can choose $C=C_{\theta}$ and $\gamma=\gamma_{\theta,*}$ for any arbitrary $\theta \in [0, 2 \pi)$. Therefore, we may completely leave aside the question of how $C_{\theta}$ and $\gamma_{\theta,*}$ depend on $\theta$. For this reason, we shall not specify the possible dependence with respect to $\theta$ of the various constants appearing in the remaining part of this text. Finally, we stress out that the function $\Phi_\theta$ does actually depend on $\theta$ through the constant $\gamma_\theta$, as it is obtained by substituting $\gamma_\theta$ for $\theta$ in the definition \eqref{es2}.

\subsection{Outline}
The remaining part of this text is organized as follows. In Sections \ref{sec-gos} and \ref{sec-proofpr1}, we build the two different types of (CGO) solutions to the BVP \eqref{eq2}, 
%\bel{eq-cgo}
%\left\{ 
%\begin{array}{rcll} 
%(-\Delta + V ) v & = & 0, & \mbox{in}\ \check{\Omega},\\ 
%v(1,\cdot) - e^{i \theta} v(0,\cdot) & = & 0, & \mbox{in}\ \omega, \\
%\pd_{x_1} v(1,\cdot) - e^{i \theta} \pd_{x_1} v(0,\cdot) & = & 0, & \mbox{in}\ \omega,
%\end{array}
%\right.
%\ee
for $\theta \in [0, 2 \pi)$, needed for the proof of Theorem \ref{thm2}, which is presented in Section \ref{sec-proofthm2}. Section \ref{sec-proofcor} contains the proof of Corollary \ref{cor-a}. Finally, a suitable characterization of the space $H_{\Delta,\theta}$, $\theta \in [0, 2 \pi)$, used in Section \ref{sec-proofpr1}, is derived in Section \ref{sec-app} in appendix.

%%%%%%%%%%%%%%%%%%%%%%%%%%%%%%%%%%%%%%%%%%%%%%%%%%%%%%%%%%%%%%%%%%%%%%%%%%%%%%%%%%%%%%%%%%%%%%%%%%%%%%%%%%%%%%%%%%%%%%%%%%%%%%%%%%%%%%%%%%%%%%%%%%%%%%%%%%%%%%%%%%%%%%%%%%%%%%%%%%%%%%%%
\section{Complex geometric optics solutions}
\label{sec-gos}
In this section we build CGO solutions to the system
\bel{eq3}
\left\{
\begin{array}{rcll} 
(-\Delta + V ) u & = & 0, & \mbox{in}\ \check{\Omega},\\ 
u(1,\cdot ) & = & e^{i\theta} u(0,\cdot), &  \mbox{on}\ \omega, \\
\pd_{x_1} u(1,\cdot ) & = & e^{i\theta} \pd_{x_1} u(0,\cdot), &  \mbox{on}\ \omega, \\
%u & = & f, & \mbox{on}\ \check{\Gamma}, %:= (0,1) \times \pd \omega,
\end{array}
\right.
\ee
where $\theta \in [0,2\pi)$ and the real valued potential $V \in L^\infty(\check{\Omega})$ are arbitrarily fixed.
More precisely, we seek a sufficiently rich set of solutions $u_\zeta$ to \eqref{eq3}, parametrized by 
\bel{e1} 
\zeta \in \cZ_\theta := \{ \zeta \in i (\theta+ 2 \pi \Z ) \times \C^2;\ | \re{\zeta} | = | \im{\zeta} | \ \mbox{and}\  \re{\zeta} \cdot  \im{\zeta} = 0  \},
\ee
of the form
\bel{e2}
u_\zeta(x)=\left( 1 + v_\zeta(x) \right) e^{{\zeta}\cdot x},\ x \in \check{\Omega},
\ee
where the behavior of the function $v_\zeta$ with respect to $\zeta$, is prescribed in a sense we shall specify further. Notice from definition \eqref{e1} that for all $\zeta \in \cZ_\theta$, we have
\bel{e3}
\Delta e^{\zeta \cdot x}=0,\ x \in \Omega. 
\ee

Actually, the analysis carried out in this text requires two different types of CGO solutions to \eqref{eq3}, denoted generically by $u_{\zeta_1}$ and $u_{\zeta_2}$, with different features we shall make precise below. The corresponding parameters $\zeta_1$ and $\zeta_2$ are the same as in \cite[section 4]{CKS2}. Namely, given $k \in \Z$ and $\eta \in \R^2\setminus \{0\}$, we pick $\xi \in \mathbb S^1$ such that $\xi \cdot \eta=0$, and we set
\bel{e4}
\ell = \ell(k,\eta,r,\theta) :=\left\{ \begin{array}{cl}  \left(\theta + 2\pi ([r]+1) \right) \left( 1 , -2 \pi k \frac{\eta}{\abs{\eta}^2} \right), & \textrm{if}\ k\ \mbox{is even},\\
\left(\theta+ 2 \pi \left( [r]+\frac{3}{2} \right) \right) \left( 1 , -2 \pi k \frac{\eta}{\abs{\eta}^2} \right), & \textrm{if}\ k\ \mbox{is odd},
\end{array}
\right.
\ee
for each $r >0$, in such a way that $\ell \cdot (2 \pi k,\eta) = \ell' \cdot \xi =0$. Here $[r]$ stands for the integer part of $r$, that is the unique integer fulfilling 
$[r] \leq r <[r]+1$, and we used the notation $\ell=(\ell_1,\ell') \in \R \times \R^2$. Next, we introduce
\bel{e5}
\tau= \tau(k,\eta,r,\theta) := \sqrt{\frac{\abs{\eta}^2}{4}+ \pi^2 k^2+\abs{\ell}^2},
\ee
and notice that
\bel{e6}
2 \pi r < \tau \leq \frac{\abs{(2 \pi k,\eta)}}{2}+4 \pi (r+1) \left( 1 + \frac{\abs{2 \pi k}}{\abs{\eta}} \right).
\ee
Thus, putting
\bel{e7}
\zeta_1 :=\left(i \pi k,-\tau \xi+i\frac{\eta}{2}\right)+i \ell\ \mbox{and}\ \zeta_2:=\left(-i \pi k, \tau \xi - i\frac{\eta}{2}\right)+ i\ell,
\ee
it is easy to see that
\bel{e8}
\zeta_1, \zeta_2 \in \cZ_\theta\ \mbox{satisfy the equation}\ \zeta_1+\overline{\zeta_2}=i(2 \pi k,\eta). 
\ee

%%%%%%%%%%%%%%%%%%%%%%%%%%%%%%%%%%%%%%%%%%%%%%%%%%%%%%%%%%%%%%%%%%%%%%%%%%
\subsection{Second order smooth CGO solutions}

The first type of CGO solutions $u_{\zeta_1}$ we need are smooth CGO functions, in the sense that the remainder term $v_{\zeta_1}$ appearing in \eqref{e2} is taken in 
$$ H_{per}^{2}(\check{\Omega}) :=\{ f \in H^2(\check{\Omega});\ f(1,\cdot)=f(0,\cdot)\ \mbox{and}\ \partial_{x_1} f(1,\cdot)=\partial_{x_1} f(0,\cdot)\ \mbox{in}\ \omega \}. $$
The existence of such functions was already established in \cite[Lemma 4.1]{CKS2}, as follows.

\begin{lemma}
\label{lm-cgo1}
Let $V \in \sV_\omega(M_\pm)$. Pick $\theta \in [0, 2 \pi)$, $k \in \Z$, and $\eta \in \R^2 \setminus \{ 0 \}$, and choose $\xi \in \mathbb S^1$ such that $\xi \cdot \eta = 0$. Let $\zeta_1$ be given by \eqref{e7}, as a function of $\tau$, defined in \eqref{e4}-\eqref{e5}.
Then, one can find $\tau_1>0$ so that for all $\tau \geq \tau_1$, there exists $v_{\zeta_1} \in H_{per}^{2}(\check{\Omega})$ such that the function $u_{\zeta_1}$, defined by \eqref{e2} with $\zeta=\zeta_1$, is solution to \eqref{eq3}. Moreover, for every $s \in [0,2]$, the estimate
\bel{e9}
\| v_{\zeta_1} \|_{H^s(\check{\Omega})} \leq C_s \tau^{s-1},
\ee
holds uniformly in $\tau \geq \tau_1$, for some constant $C_s>0$, depending only on $s$, $\omega$, and $M_\pm$.
\end{lemma}

The second type of test functions needed for the proof of Theorem \ref{thm1} are CGO solutions to \eqref{eq3}, vanishing on a suitable subset of the boundary $(0,1) \times \pd \omega$.
They are described in Subsection \ref{sec-cgoDBC}.

%%%%%%%%%%%%%%%%%%%%%%%%%%%%%%%%%%%%%%%%%%%%%%%%%%%%%%%%%%%%%%%%%%%%%%%%%%
\subsection{CGO solutions vanishing on a sub-part of the boundary}
\label{sec-cgoDBC}
For $\xi \in \mathbb S^1$ and $\epsilon>0$, we set
\bel{e9b} 
\pd \omega_{\epsilon,\xi}^+:=\{ x' \in \pd \omega;\  \xi \cdot \nu'(x') > \epsilon\}\ \mbox{and}\ \pd \omega_{\epsilon,\xi}^-:=\{ x' \in \pd \omega;\  \xi \cdot \nu'(x') \leq \epsilon \},
\ee
and we write $\check{\Gamma}_{\epsilon,\xi}^\pm$ instead of $(0,1) \times \pd \omega_{\epsilon,\xi}^\pm$, in the sequel.

Bearing in mind that $F'$ is a closed neighborhood in $\pd \omega$, of the subset $\pd \omega_{\xi_0}^+$ defined by \eqref{a0b}, we pick $\epsilon>0$ so small that
\bel{e10}
\forall \xi \in \mathbb S^1,\ \left( |\xi-\xi_0| \leq \epsilon \right) \Longrightarrow \left( \pd \omega_{\epsilon,-\xi}^- \subset F' \right).
\ee
In this subsection, we aim for building solutions $u \in \cH_{\Delta,\theta}(\check{\Omega})$ of the form \eqref{e2} with $\zeta=\zeta_2$, where $\zeta_2$ is given by \eqref{e4}-\eqref{e5} and \eqref{e7}, to the BVP
\bel{eq4}
\left\{
\begin{array}{rcll} 
(-\Delta + V ) u & = & 0, & \mbox{in}\ \check{\Omega},\\ 
u(1,\cdot ) & = & e^{i\theta} u(0,\cdot), &  \mbox{on}\ \omega, \\
\pd_{x_1} u(1,\cdot ) & = & e^{i\theta} \pd_{x_1} u(0,\cdot), &  \mbox{on}\ \omega, \\
u & = & 0, & \mbox{on}\ \check{\Gamma}_{\frac{\epsilon}{2},-\xi}^+.
\end{array}
\right.
\ee

%Moreover, we shall impose that $v_{\zeta_2}=e^{-\zeta_2 \cdot x} u_{\zeta_2} - 1$ satisfy the estimate
%\bel{e13}
%\| v_{\zeta_2} \|_{L^2(\check{\Omega})}\leq C \tau^{-\frac{1}{2}},
%\ee
%for some constant $C>0$, independent of $\tau$.

The result we have in mind is as follows.

\begin{proposition}
\label{pr1} 
Let $V$, $\theta$, $k$, $\eta$, $\xi$, and $\tau$, be the same as in Lemma \ref{lm-cgo1}, and let 
$\zeta_2$ be defined by \eqref{e7}. Then, one can find $\tau_2>0$, so that for each $\tau \geq \tau_2$, there exists $v_{\zeta_2} \in \cH_{\Delta,0}(\check{\Omega})$ such that the function $u_{\zeta_2}$, defined by \eqref{e2} with $\zeta=\zeta_2$, is a $\cH_{\Delta,\theta}(\check{\Omega})$-solution to \eqref{eq4}. Moreover, $v_{\zeta_2}$ satisfies
\bel{e13}
\| v_{\zeta_2} \|_{L^2(\check{\Omega})}\leq C \tau^{-\frac{1}{2}},\ \tau \geq \tau_2,
\ee
for some constant $C>0$ depending only on $\omega$, $M_+$ and $F'$.
\end{proposition}

The proof of Proposition \ref{pr1} is postponed to Section \ref{sec-proofpr1}. 

We notice from the last line in \eqref{eq4}, that the solution $u_{\zeta_2}$ given by 
Proposition \ref{pr1} for $\tau \geq \tau_2$, verifies $\cT_{0,\theta}u_{\zeta_2}=0\ \mbox{on}\ \check{\Gamma} \setminus \check{F}$, as we have
$\check{\Gamma}\setminus \check{F} \subset \check{\Gamma} \setminus \check{\Gamma}_{\epsilon,-\xi}^- \subset
\check{\Gamma}_{\frac{\epsilon}{2},-\xi}^+$.
Therefore, it holds true that
\bel{e11}
\cT_{0,\theta}u_{\zeta_2} \in \sH_{c,\theta}(\check{F}),\ \tau \geq \tau_2.
\ee
Further, we have
\bel{e14} 
\| \cT_{0,\theta} u_{\zeta_2} \|_{\sH_\theta(\check{\Gamma})} \leq  %C \| u_{\zeta_2} \|_{H_\Delta(\check{\Omega})}
C \left( \| u_{\zeta_2} \|_{L^2(\check{\Omega})}
+ \| V u_{\zeta_2} \|_{L^2(\check{\Omega})} \right) \leq C(1+M_+) \| u_{\zeta_2} \|_{L^2(\check{\Omega})}, 
\ee
for some positive constant $C=C(\omega)$, by the continuity of $\cT_{0,\theta} : H_{\Delta,\theta}(\check{\Omega}) \to \sH_\theta(\check{\Gamma})$ and the first line of \eqref{eq4}.
As 
\bel{e15}
\zeta_2 \cdot x - \tau \xi \cdot x' \in i \R,\ x=(x_1,x') \in \R \times \R^2,
\ee
by \eqref{e4}-\eqref{e5} and \eqref{e7}, we infer from the identity $u_{\zeta_2} = e^{\zeta_2 \cdot x} (1+v_{\zeta_2})$, that
$$
\| u_{\zeta_2} \|_{L^2(\check{\Omega})}  \leq  \| e^{\zeta_2 \cdot x}  \|_{L^2(\check{\Omega})} + \| e^{\zeta_2 \cdot x} v_{\zeta_2} \|_{L^2(\check{\Omega})} 
\leq \| e^{\tau \xi \cdot x'} \|_{L^2(\check{\Omega})} + \| e^{\tau \xi \cdot x'} v_{\zeta_2} \|_{L^2(\check{\Omega})}, 
$$
and hence that
$$
\| u_{\zeta_2} \|_{L^2(\check{\Omega})}  \leq e^{c_\omega \tau} \left( \mbox{meas}(\omega) + \| v_{\zeta_2} \|_{L^2(\check{\Omega})}  \right),
$$
where $c_\omega:=\sup \{ | x'|,\ x' \in \omega \}$ and $\mbox{meas}(\omega)$ denotes the two-dimensional Lebesgue measure of $\omega$.
%From this, the first line of \eqref{eq4}, and \eqref{e13}, it then follows that
%\bel{e15b}
%\| u_{\zeta_2} \|_{H_\Delta(\check{\Omega})}  \leq C e^{c_\omega \tau},\ \tau \geq \tau_2,
%$$
%for some constant $C=C(\omega,F',M_+)>0$ is independent of $\tau$. Finally, putting \eqref{e14} together with \eqredf{e15b} we get that
From this, \eqref{e13} and \eqref{e14}, it then follows that
\bel{e16}
\| \cT_{0,\theta} u_{\zeta_2} \|_{\sH_\theta(\check{\Gamma})} \leq C e^{c_\omega \tau},\ \tau \geq \tau_2, 
\ee
where the constant $C>0$, depends only on $\omega$, $M_+$, and $F'$.

%%%%%%%%%%%%%%%%%%%%%%%%%%%%%%%%%%%%%%%%%%%%%%%%%%%%%%%%%%%%%%%%%%%%%%%%%%%%%%%%%%%%%%%%%%%%%%%%%%%%%%%%%%%%%%%%%%%%%%%%%%%%%%%%%%%%%%%%%%%%%%%%%%%%%%%%%%%%%%%%%%%%%%%%%%%%%%%%%%%%%%%%
\section{Proof of Proposition \ref{pr1}}
\label{sec-proofpr1}
The  proof of Proposition \ref{pr1} is by means of two technical results, given in Subsection \ref{sec-tools}. The first one is a Carleman estimate for the quasi-periodic Laplace operator $-\Delta+V$ in $L^2(\check{\Omega})$, which was inspired by \cite[Lemma 2.1]{BU}. The second one is an existence result for the BVP \eqref{eq2} with non zero source term.

%%%%%%%%%%%%%%%%%%%%%%%%%%%%%%%%%%%%%%%%%%%%%%%%%%%%%%%%%%%%%%%%%%%%%%%%%%
\subsection{Two useful tools}
\label{sec-tools}
We start by recalling the following Carleman inequality, borrowed from \cite[Corollary 5.2]{CKS2}.

\begin{lemma}
\label{lm-ce}
For $M>0$ arbitrarily fixed, let $V \in L^\infty(\Omega)$ satisfy \eqref{a0} and $\| V \|_{L^{\infty}(\Omega)} \leq M$.
Then, there exist two constants $C>0$ and $\tau_0>0$, both of them depending only on $\omega$ and $M$, such that for all $\xi \in \mathbb S^1$, all $\theta \in [0,2 \pi)$, and all $w \in C^2_\theta([0,1] \times \overline{\omega})$ obeying $w_{\vert \check{\Gamma}}=0$, we have
\bea
& & C \| e^{-\tau \xi \cdot x'} w \|_{L^2(\check{\Omega})}^2 + \tau \| e^{-\tau \xi \cdot x'} (\xi \cdot \nu')^{\frac{1}{2}} \pd_\nu w \|_{L^2(\check{\Gamma}_\xi^+)}^2
\nonumber \\
& \leq & \| e^{-\tau \xi \cdot x'} (-\Delta+V) w \|_{L^2(\check{\Omega})}^2 + \tau  \| e^{-\tau \xi \cdot x'} |\xi \cdot \nu'|^{\frac{1}{2}} \pd_\nu w \|_{L^2(\check{\Gamma}_\xi^-)}^2, \label{ce}
\eea
provided $\tau \geq \tau_0$.
\end{lemma}

Here and henceforth, we write $\check{\Gamma}_\xi^\pm$ instead of $(0,1) \times \pd \omega_{\xi}^\pm$.

Armed with Lemma \ref{lm-ce}, we turn now to establishing the following existence result for the BVP \eqref{eq2} with non zero source term.

\begin{lemma}
\label{lm2}
For $M>0$, let $V \in L^{\infty}(\check{\Omega})$ be real valued and such that $\| V \|_{L^{\infty}(\check{\Omega})} \leq M$, and let
$\theta \in [0,2 \pi)$, $\xi \in \mathbb S^1$, $f \in L^2(\check{\Omega})$, and $g \in L^2(\check{\Gamma}_\xi^- \setminus \check{\Gamma}_\xi^+;\ | \xi \cdot \nu' |^{-\frac{1}{2}} dx)$.
Then, for every $\tau \in [\tau_0,+\infty)$, where $\tau_0$ is the same as in Lemma \ref{lm-ce}, there exists $v \in H_{\Delta,\theta}(\check{\Omega})$ fulfilling
\bel{i1}
\left\{
\begin{array}{rcll} 
(-\Delta + V ) v & = & f, & \mbox{in}\ \check{\Omega},\\ 
%v(1,\cdot) & = & e^{i\theta} v(0,\cdot), &  \mbox{on}\ \omega, \\
%\pd_{x_1} v(1,\cdot ) & = & e^{i\theta} \pd_{x_1} v(0,\cdot), &  \mbox{on}\ \omega, \\
v & = & g, & \mbox{on}\ \check{\Gamma}_\xi^- \setminus \check{\Gamma}_\xi^+.
\end{array}
\right.
\ee
Moreover, $v$ satisfies the estimate
\bel{i2}
\| e^{-\tau \xi \cdot x'} v \|_{L^2(\check{\Omega})} \leq C \left( \tau^{-1}  \| e^{-\tau \xi \cdot x'} f \|_{L^2(\check{\Omega})} + \tau^{-\frac{1}{2}}  
\| e^{-\tau \xi \cdot x'} | \xi \cdot \nu' |^{-\frac{1}{2}} g \|_{L^2(\check{\Gamma}_\xi^- \setminus \check{\Gamma}_\xi^+)} \right),
\ee
for some constant $C>0$, depending only on $\omega$ and $M$.
\end{lemma}
\begin{proof}
We denote by $C_{0,\theta}^2([0,1] \times \overline{\omega})$ the set of $C_{\theta}^2([0,1] \times \overline{\omega})$-functions vanishing on the boundary $\check{\Gamma}$, i.e.
$$ C_{0,\theta}^2([0,1] \times \overline{\omega}):=\left\{ w \in C^2_\theta \left( [0,1] \times\overline{\omega} \right);\ w_{\vert \check{\Gamma}}=0 \right\}. $$
By substituting $(-\xi)$ for $\xi$ in the Carleman estimate \eqref{ce}, we get for all $w \in C_{0,\theta}^2([0,1] \times \overline{\omega})$, that
\bea
& & \tau \norm{e^{\tau \xi \cdot x'} w}_{L^2(\check{\Omega})} + \tau^{\frac{1}{2}} \norm{e^{\tau \xi \cdot x'} | \xi \cdot \nu' |^{\frac{1}{2}} \pd_\nu w}_{L^2(\check{\Gamma}_\xi^-)}
\nonumber \\
& \leq & C \left( \norm{e^{\tau \xi \cdot x'}  (-\Delta+V) w}_{L^2(\check{\Omega})} +  \norm{e^{\tau \xi \cdot x'} (\tau \xi \cdot \nu' )^{\frac{1}{2}} \pd_\nu w}_{L^2(\check{\Gamma}_\xi^+)} \right). \label{i3}
\eea
Put $\cM_\theta := \varrho \left( C_{0,\theta}^2([0,1] \times \overline{\omega}) \right) = \left\{ \varrho(w); w \in C_{0,\theta}^2 \left( [0,1] \times \overline{\omega} \right) \right\}$, where $\varrho(w) := \left( (-\Delta + V) w , \pd_\nu w_{\vert \check{\Gamma}_\xi^+} \right)$.
Since $\varrho$ is one-to-one from $C_{0,\theta}^2 \left( [0,1] \times \overline{\omega} \right)$ onto $\cM_\theta$, according to \eqref{i3}, then the antilinear form
%$$\Upsilon : \left( (-\Delta+V) w ,\pd_\nu w_{\vert \check{\Gamma}_\xi^+} \right) \mapsto 
%\langle w, f \rangle_{L^2(\check{\Omega})}-\langle \pd_\nu w , g \rangle_{L^2(\check{\Gamma}_\xi^-)}, $$
\bel{i3b}
\Upsilon : \varrho(w) \mapsto 
\langle f , w \rangle_{L^2(\check{\Omega})}-\langle g , \pd_\nu w \rangle_{L^2(\check{\Gamma}_\xi^- \setminus \check{\Gamma}_\xi^+)},
\ee
is well defined on $\cM_\theta$, regarded as a subspace of $\sL_{\tau,\frac{1}{2}}:=L^2(\check{\Omega} ; e^{\tau \xi \cdot x'} dx) \times L^2(\check{\Gamma}_\xi^+;e^{\tau \xi \cdot x'} |\tau \xi \cdot \nu'(x')|^{\frac{1}{2}} dx)$.
Moreover, for each $w \in C_{0,\theta}^2([0,1] \times \overline{\omega})$, we have
\beas
%& & \left| \Upsilon \left( (-\Delta+V) w,\pd_\nu w_{\vert \check{\Gamma}_\xi^+}) \right) \right| \nonumber \\
\left| \Upsilon \left( \varrho(w) \right) \right| & \leq & 
\norm{e^{-\tau \xi \cdot x'} f}_{L^2(\check{\Omega})} \norm{e^{\tau \xi \cdot x'} w}_{L^2(\check{\Omega})}+ 
\norm{e^{-\tau \xi \cdot x'} | \xi \cdot \nu' |^{-\frac{1}{2}} g}_{L^2(\check{\Gamma}_\xi^- \setminus \check{\Gamma}_\xi^+)} \norm{e^{\tau \xi \cdot x'} | \xi \cdot \nu' |^{\frac{1}{2}} \pd_\nu w}_{L^2(\check{\Gamma}_\xi^-)}\\
& \leq &  \left( \tau^{-1} \norm{e^{-\tau \xi \cdot x'} f}_{L^2(\check{\Omega})} + \tau^{-\frac{1}{2}} \norm{e^{-\tau \xi \cdot x'} | \xi \cdot \nu' |^{-\frac{1}{2}} g}_{L^2(\check{\Gamma}_\xi^- \setminus \check{\Gamma}_\xi^+)} \right)  \\
& & \times \left( \tau \norm{e^{\tau \xi \cdot x'} w}_{L^2(\check{\Omega})} + \tau^{\frac{1}{2}} \norm{e^{\tau \xi \cdot x'} | \xi \cdot \nu' |^{\frac{1}{2}} \pd_\nu w}_{L^2(\check{\Gamma}_\xi^-)} \right) \\
& \leq & C \left( \tau^{-1} \norm{e^{-\tau \xi \cdot x'} f}_{L^2(\check{\Omega})} + \tau^{-\frac{1}{2}} \norm{e^{-\tau \xi \cdot x'} | \xi \cdot \nu' |^{-\frac{1}{2}} g}_{L^2(\check{\Gamma}_\xi^- \setminus \check{\Gamma}_\xi^+)} \right) \\
& &  \times \left( \norm{e^{\tau \xi \cdot x'}  (-\Delta+V) w}_{L^2(\check{\Omega})} +  \norm{e^{\tau \xi \cdot x'} | \tau \xi \cdot \nu' |^{\frac{1}{2}} \pd_\nu w}_{L^2(\check{\Gamma}_\xi^+)} \right),
\eeas
by \eqref{i3}, and hence
$$
\left| \Upsilon \left( \varrho(w) \right) \right| \leq 
C \left( \tau^{-1} \norm{e^{-\tau \xi \cdot x'} f}_{L^2(\check{\Omega})} + \tau^{-\frac{1}{2}} \norm{e^{-\tau \xi \cdot x'} | \xi \cdot \nu' |^{-\frac{1}{2}} g}_{L^2(\check{\Gamma}_\xi^- \setminus \check{\Gamma}_\xi^+)} \right) \| \varrho(w) \|_{\sL_{\tau,\frac{1}{2}}},
$$
where $C$ is the same constant as in \eqref{i3}.
Thus, by Hahn Banach's theorem, $\Upsilon$ extends to an antilinear form on $\sL_{\tau,\frac{1}{2}}$, still denoted by $\Upsilon$, satisfying
\bel{i5}
\norm{\Upsilon}\leq C \left( \tau^{-1} \norm{e^{-\tau \xi \cdot x'} f}_{L^2(\check{\Omega})} + \tau^{-\frac{1}{2}} \norm{e^{-\tau \xi \cdot x'} | \xi \cdot \nu' |^{-\frac{1}{2}} g}_{L^2(\check{\Gamma}_\xi^- \setminus \check{\Gamma}_\xi^+)} \right).
\ee
Therefore, there exists $(v ,\tilde{g} )$ in $\sL_{-\tau,-\frac{1}{2}} := L^2(\check{\Omega} ; e^{-\tau \xi \cdot x'} dx) \times L^2(\check{\Gamma}_\xi^+;e^{-\tau \xi \cdot x'} |\tau \xi \cdot \nu'(x')|^{-\frac{1}{2}} dx)$, the dual space to $\sL_{\tau,\frac{1}{2}}$ with pivot space
$L^2(\check{\Omega}) \times L^2(\check{\Gamma}_\xi^+)$, such that
\bel{i5b}
\Upsilon \left( \varrho (w) \right ) = \langle (v ,\tilde{g} ) , \varrho (w) \rangle_{L^2(\check{\Omega}) \times L^2(\check{\Gamma}_\xi^+)} =
\langle v , (-\Delta+V) w \rangle_{L^2(\check{\Omega})} + \langle \tilde{g} , \pd_\nu w \rangle_{L^2(\check{\Gamma}_\xi^+)},\ w \in C_{0,\theta}^2([0,1] \times \overline{\omega}).
\ee
This and \eqref{i3b} yield that
\bel{i6}
\langle v , (-\Delta+V) w \rangle_{L^2(\check{\Omega})} + \langle \tilde{g} , \pd_\nu w \rangle_{L^2(\check{\Gamma}_\xi^+)}
=
\langle f , w \rangle_{L^2(\check{\Omega})}-\langle g , \pd_\nu w  \rangle_{L^2(\check{\Gamma}_\xi^- \setminus \check{\Gamma}_\xi^+)},\ w \in C_{0,\theta}^2([0,1] \times \overline{\omega}).
\ee
Taking $w \in C_0^\infty(\check{\Omega}) \subset C_{0,\theta}^2([0,1] \times \overline{\omega})$ in \eqref{i6}, we get that $(-\Delta + V) v = f$ in the distributional sense on $\check{\Omega}$. Moreover, since $f \in L^2(\check{\Omega} ; e^{-\tau \xi \cdot x'} dx)$ and $L^2(\check{\Omega} ; e^{-\tau \xi \cdot x'} dx) \subset L^2(\check{\Omega})$, then the identity $(-\Delta + V) v = f$ holds in $L^2(\check{\Omega})$. This entails that $v \in H_{\Delta}(\check{\Omega})$.

We turn now to proving that $v \in H_{\Delta,\theta}(\check{\Omega})$. With reference to Lemma \ref{lm-a} in Appendix, it suffices to show that
\bel{i7}
\sum_{k \in \Z} \left\| (\Delta'-(\theta + 2 k \pi)^2) \hat{v}_{k,\theta} \right\|_{L^2(\omega)}^2 < \infty,
\ee
where 
\bel{i8} 
\hat{v}_{k,\theta}(x') := \langle v(\cdot,x') , \varphi_{k,\theta} \rangle_{L^2(0,1)},\ x' \in \omega,\ k \in \Z,
\ee
and
\bel{i9}
\varphi_{k,\theta}(x_1):=e^{i (\theta + 2 k \pi) x_1},\ x_1 \in (0,1),\ k \in \Z.
\ee
To do that, we fix $k \in \Z$, pick $\chi \in C_0^{\infty}(\omega)$, and apply \eqref{i6} with $w(x)=\varphi_{k,\theta}(x_1) \chi(x')$, getting
\bel{i10}
\langle v, \varphi_{k,\theta} \left( -\Delta' + (\theta + 2 k \pi)^2 \right) \chi \rangle_{L^2(\check{\Omega})} =
\langle h , \varphi_{k,\theta} \chi \rangle_{L^2(\check{\Omega})},
\ee
with $h:=f - V v$. Next, by Fubini's theorem, we have
\bea
\langle v , \varphi_{k,\theta} \left( -\Delta' + (\theta+ 2 k \pi)^2 \right) \chi \rangle_{L^2(\check{\Omega})}
& = & \int_\omega \left( \int_0^1 v(x_1,x') \overline{\varphi_{k,\theta}(x_1)} dx_1 \right) \left( -\Delta' + (\theta + 2 k \pi )^2 \right) \overline{\chi(x')} dx' \nonumber \\
& = & \langle \hat{v}_{k,\theta} ,\left( -\Delta' + (\theta + 2 k \pi)^2 \right) \chi \rangle_{L^2(\omega)} \nonumber \\
& = & \langle \left( -\Delta' + (\theta + 2 k \pi)^2 \right) \hat{v}_{k,\theta} , \chi \rangle_{(C_0^\infty)'(\omega) , C_0^\infty(\omega)}, \label{i11}
\eea
and similarly
\bel{i12} 
\langle h , \varphi_{k,\theta} \chi \rangle_{L^2(\check{\Omega})} = \langle \hat{h}_{k,\theta} , \chi \rangle_{(C_0^\infty)'(\omega) , C_0^\infty(\omega)}.
\ee
Putting \eqref{i10}--\eqref{i12} together, we find that
\bel{i13}
\left( -\Delta' + (\theta + 2 k \pi)^2 \right) \hat{v}_{k,\theta} = \hat{h}_{k,\theta},\ k \in \Z. 
\ee
Since $\{ \varphi_{k,\theta},\ k \in \Z \}$ is an Hibertian basis of $L^2(0,1)$ and $L^2(\check{\Omega}) = L^2(0,1;L^2(\omega))$, then $w \mapsto \{ \hat{w}_{k,\theta},\ k \in \Z \}$ is a unitary transform of $L^2(\check{\Omega})$ onto $\bigoplus_{k \in \Z} L^2(\omega)$. From this, \eqref{i13}, and the fact that $h \in L^2(\check{\Omega})$, it then follows that
$$ \sum_{k \in \Z} \left\| \left( -\Delta' + (\theta + 2 k \pi)^2 \right) \hat{v}_{k,\theta} \right\|_{L^2(\omega)}^2 = \sum_{k \in \Z} \| \hat{h}_{k,\theta} \|_{L^2(\omega)}^2 = \| h \|_{L^2(\check{\Omega})}^2, $$
which entails \eqref{i7}.

Further, as $v \in H_{\Delta,\theta}(\check{\Omega})$, we infer from the Green formula that
\bel{i14}
\langle v , (-\Delta+V) w \rangle_{L^2(\check{\Omega})} = \langle (-\Delta+V)  v , w \rangle_{L^2(\check{\Omega})} - \langle v , \pd_\nu w \rangle_{L^2(\check{\Gamma})},\ 
w \in C_{0,\theta}^2([0,1] \times \overline{\omega}).
\ee
Bearing in mind hat $(-\Delta+V)  v=f$, it follows from \eqref{i6} and \eqref{i14} that
\bel{i15}
\langle v , \pd_\nu w \rangle_{L^2(\check{\Gamma})} = \langle g , \pd_\nu w  \rangle_{L^2(\check{\Gamma}_\xi^- \setminus \check{\Gamma}_\xi^+)} + \langle \tilde{g} , \pd_\nu w \rangle_{L^2(\check{\Gamma}_\xi^+)},\
w \in C_{0,\theta}^2([0,1] \times \overline{\omega}).
\ee
Since $w$ is arbitrary in $C_{0,\theta}^2([0,1] \times \overline{\omega})$, and $\check{\Gamma}=\left( \check{\Gamma}_\xi^- \setminus \check{\Gamma}_\xi^+ \right) \cup  \check{\Gamma}_\xi^+$, we deduce from \eqref{i15} that $v=g$ on $\check{\Gamma}_\xi^- \setminus \check{\Gamma}_\xi^+$, and $v=\tilde{g}$ on $\check{\Gamma}_\xi^+ $.

Last, upon plugging the identity $v=\tilde{g}$ on $\check{\Gamma}_\xi^+$, into \eqref{i5b}, we find that
$$
\Upsilon \left( \varrho (w) \right ) =
\langle v , (-\Delta+V) w \rangle_{L^2(\check{\Omega})} + \langle v , \pd_\nu w \rangle_{L^2(\check{\Gamma}_\xi^+)} 
= \left \langle \left( v , v_{| \check{\Gamma}_\xi^+} \right) , \varrho (w) \right \rangle_{L^2(\check{\Omega}) \times L^2(\check{\Gamma}_\xi^+)},\ w \in C_{0,\theta}^2([0,1] \times \overline{\omega}),
$$
which may be equivalently rewritten as
$$ \Upsilon \left( h \right ) = \left \langle \left( v , v_{| \check{\Gamma}_\xi^+} \right) , h \right \rangle_{\sL_{-\tau,-\frac{1}{2}}, \sL_{\tau,\frac{1}{2}}},\ h \in \cM_\theta. $$
As a consequence we have
$\| \Upsilon \| = \left\|  \left( v , v_{| \check{\Gamma}_\xi^+} \right) \right\|_{\sL_{-\tau,-\frac{1}{2}}} = \| e^{-\tau \xi \cdot x'} v \|_{L^2(\check{\Omega})} + \| e^{-\tau \xi \cdot x'} | \tau \xi \cdot \nu' |^{-\frac{1}{2}} \pd_\nu v \|_{L^2(\check{\Gamma}_\xi^+)}$, and hence
$$  \| e^{-\tau \xi \cdot x'} v \|_{L^2(\check{\Omega})} \leq \| \Upsilon \|. $$
Putting this together with \eqref{i5}, we end up getting \eqref{i2}.
\end{proof}

Armed with Lemma \ref{lm2} we are in position to complete the proof of Proposition \ref{pr1}.

\subsection{Completion of the proof}
Let us first notice that $u_{\zeta_2}=e^{\zeta_2 \cdot x}(1+v_{\zeta_2})$ is a solution to \eqref{eq4} if and only if $v:= e^{\zeta_2\cdot x} v_{\zeta_2}$ is a solution to the BVP
\bel{p1}
\left\{
\begin{array}{rcll} 
(-\Delta + V ) v & = & f, & \mbox{in}\ \check{\Omega},\\ 
v(1,\cdot ) & = & e^{i \theta} v(0,\cdot), &  \mbox{on}\ \omega, \\
\pd_{x_1} v(1,\cdot ) & = & e^{i\theta} \pd_{x_1} v(0,\cdot), &  \mbox{on}\ \omega, \\
v & = & -e^{\zeta_2 \cdot x}, & \mbox{on}\ \check{\Gamma}_{\frac{\epsilon}{2},-\xi}^+,
\end{array}
\right.
\ee
with $f(x):=-V(x)e^{\zeta_2\cdot x}$. %The next step is to see that \eqref{p1} can be brought into the form of the BVP \eqref{i1} upon taking

Next, we take $\psi$ in $C^\infty_0(\R^2)$ such that $\mbox{supp}\ \psi \cap \pd \omega \subset \pd \omega_{\frac{\epsilon}{3},-\xi}^+$ and $\psi(x')=1$ for $x' \in \pd \omega_{\frac{\epsilon}{2},-\xi}^+$. Since $g(x):=-e^{\zeta_2 \cdot x} \psi(x')$ vanishes for every $x \in \check{\Gamma}_{\frac{\epsilon}{3},-\xi}^-$, then we see that $|\xi \cdot \nu'|^{-\frac{1}{2}} g \in L^2(\check{\Gamma}_\xi^- \setminus \check{\Gamma}_\xi^+)$. Therefore, we may apply Lemma \ref{lm2} for the above defined functions $f$ and $g$. We get that \eqref{i1} admits a solution $v \in H_{\Delta,\theta}(\check{\Omega})$, which is obviously a solution to \eqref{p1} as well. 

Further, remembering \eqref{e16}, we find by combining the identity $v_{\zeta_2}=e^{-\zeta_2\cdot x} v$ with the estimate \eqref{i2}, that
\beas
\| v_{\zeta_2} \|_{L^2(\check{\Omega})}   % = \| e^{-\tau \xi \cdot x'} e^{\tau \xi \cdot x'-\zeta_2\cdot x} v \|_{L^2(\check{\Omega})} \\
= \| e^{-\tau \xi \cdot x'} v \|_{L^2(\check{\Omega})} 
& \leq & C \left( \tau^{-1} \| V \|_{L^2(\check{\Omega})} + \tau^{-\frac{1}{2}} \| | \xi \cdot \nu' |^{-\frac{1}{2}} \psi \|_{L^2(\check{\Gamma}_\xi^-)} \right) \\
& \leq & C \tau^{-\frac{1}{2}} \left( \tau^{-\frac{1}{2}} M \mbox{meas}(\omega) + \epsilon^{-\frac{1}{2}} \|  \psi \|_{L^2(\check{\Gamma}_\xi^-)} \right).
\eeas
This yields \eqref{e13} and terminates the proof of Proposition \ref{pr1}.

\section{Proof of  Theorem \ref{thm2}}
\label{sec-proofthm2}
The derivation of Theorem \ref{thm2} follows the same path as the proof of \cite[Theorem 3.3]{CKS2}, the only difference being the introduction of the CGO solutions described by Proposition \ref{pr1}. The main technical task here is to establish the estimate \eqref{f0} stated in Lemma \ref{lm-pre}, below, which is similar to \cite[Lemma 6.1]{CKS2}. 

To this end, we start by setting the working parameter $\epsilon$, appearing in \eqref{e10}. More precisely, since $G'$ is a closed neighborhood of $\pd \omega_{\xi_0}^-$, we assume upon possibly shortening $\epsilon>0$, that
\bel{f1}
\forall \xi \in \mathbb S^1,\ \left( |\xi-\xi_0| \leq \epsilon \right) \Longrightarrow \left( \pd \omega_{\epsilon,-\xi}^- \subset F'\ \mbox{and}\ \pd \omega_{\epsilon,\xi}^- \subset G' \right).
\ee
Next, we refer to \eqref{e4}-\eqref{e5} and choose $r_*>0$ so large that 
\bel{f3} 
\left( r \geq r_* \right) \Longrightarrow \left( \tau \geq \max \{ \tau_j,\ j=0,1,2 \} \right),
\ee
where $\tau_0$, $\tau_1$, and $\tau_2$, are the same as in Lemma \ref{lm-ce}, Lemma \ref{lm-cgo1}, and Proposition \ref{pr1}, respectively. Thus, putting
$$ V(x):= \left\{ \begin{array}{cl} (V_2-V_1)(x) & \mbox{if}\ x \in \check{\Omega} \\ 0, & \mbox{if}\ x \in (0,1) \times (\R^2 \setminus \omega), \end{array} \right. $$
we have the following result.

\begin{lemma}
\label{lm-pre}
Let $\epsilon$ be defined by \eqref{f1} and let $r_*$ be as in \eqref{f3}. Then the estimate
\bel{f0}
\left| \int_{(0,1) \times \R^2} V(x_1,x') e^{-i(2 \pi kx_1+\eta\cdot x')} dx_1 d_x' \right|^2 \leq C \left(\frac{1}{\tau}+e^{C' \tau}\norm{\Lambda_{{V}_2,\theta}-\Lambda_{{V}_1,\theta}}^2\right),
\ee
holds for all $r \geq r_*$, all $\xi \in \mathbb S^1$ such that $\abs{\xi-\xi_0} \leq \epsilon$, all $\eta \in \R^2 \setminus \{0\}$ satisfying $\xi \cdot \eta=0$, and all $k \in \mathbb Z$.
Here the positive constants $C$ and $C'$ depend only on $\omega$, $M_\pm$, $F'$, $G'$, and $\xi_0$.
\end{lemma}
\begin{proof}
We define $\zeta_j$, $j=1,2$, as in \eqref{e7} and, we denote by 
\bel{f3b}
u_{\zeta_1}(x)=( 1+ v_{\zeta_1}(x) ) e^{\zeta_1 \cdot x},\ x \in \check{\Omega},
\ee
the $\cH_{\theta}^2(\check{\Omega})$-solution to \eqref{eq3} associated with $V=V_1$, which is given by Lemma \ref{lm-cgo1}. 
%We recall for further reference that $v_{\zeta_1}$ satisfies the estimate \eqref{e9} with $s=1$, entailing
%\bel{f4}
%\|  v_{\zeta_1 }\|_{H^1(\check{\Omega})} \leq C,
%\ee
%for some constant $C=C(\omega,M_+)>0$. 
Similarly, we denote by 
\bel{f4b}
u_{\zeta_2}(x)=( 1+ v_{\zeta_2}(x) ) e^{\zeta_2 \cdot x},\ x \in \check{\Omega},
\ee
the $\cH_{\Delta,\theta}(\check{\Omega})$-solution to \eqref{eq4}, where $V_2$ is substituted for $V$, defined by Proposition \ref{pr1}. 

Then we notice that if $w_1$ solves the BVP
\bel{f5}
\left\{
\begin{array}{rcll}
 (-\Delta + V_1 )w_1 & = & 0, & \mbox{in}\ \check{\Omega},
\\
w_1 & = & \mathcal T_{0,\theta} u_{\zeta_2}, & \mbox{on}\ \check{\Gamma}, \\
w_1 (1,\cdot) - e^{i \theta} w_1(0,\cdot) & = & 0, & \mbox{on}\ \omega, \\
\pd_{x_1} w_1 (1,\cdot) - e^{i \theta} \pd_{x_1} w_1(0,\cdot) & = & 0, & \mbox{on}\ \omega,
\end{array}
\right.
\ee
then $u:=w_1-u_{\zeta_2}$ is solution to the system
\bel{f6}
\left\{ 
\begin{array}{rcll}
(-\Delta + V_1) u & = & V u_{\zeta_2}, & \mbox{in}\ \check{\Omega}, \\
u & = & 0, &\mbox{on}\ \check{\Gamma}, \\
u(1,\cdot) - e^{i\theta} u(0,\cdot) & = & 0, & \mbox{on}\ \omega, \\
\pd_{x_1} u(1,\cdot) - e^{i\theta} \pd_{x_1} u(0,\cdot) & = & 0, & \mathrm{on}\ \omega. 
\end{array}
\right.
\ee
%As $V u_{\zeta_2}\in L^2(\check{\Omega})$ and $\mathcal A_{V_1,\theta}$ is boundedly invertible in $L^2(\check{\Omega})$, with $\| \mathcal A_{V_1,\theta}^{-1} %\|_{\cB(L^2(\check{\Omega}))} \leq (C_\omega - M_-)^{-1}$, then $u=\mathcal A_{V_1,\theta}^{-1} V u_{\zeta_2} \in \cH^2_\theta(\check{\Omega})$ satisfies the estimate 
%$$ \norm{u}_{H^2(\check{\Omega})} \leq \frac{M_+}{C_\omega-M_-} \norm{u_{\zeta_2}}_{L^2(\check{\Omega})}.$$ 
Thus, taking into account that $(-\Delta + V_1) u_{\zeta_1}=0$ in $\check{\Omega}$, we deduce from \eqref{f6} and the Green formula that
$$
\int_{\check{\Omega}} V u_{\zeta_2} \overline{u_{\zeta_1}} d x =\int_{\check{\Omega}} (-\Delta+V_1)u \overline{u_{\zeta_1}}d x
%= -\int_{\check{\Omega}} u \overline{(-\Delta+V_1)u_{\zeta_1}} d x + \int_{\check{\Gamma}} (\partial_\nu u) \overline{u_{\zeta_1}} d\sigma(x)
= \int_{\check{\Gamma}} (\pd_\nu u) \overline{u_{\zeta_1}} d\sigma(x),
$$
which can be equivalently rewritten with the help of \eqref{e9b}, as
\bel{f7}
\int_{\check{\Omega}} Vu_{\zeta_2}\overline{u_{\zeta_1}} d x
=\int_{\Gamma_{\xi,\epsilon}^+} (\pd_\nu u ) \overline{u_{\zeta_1}} d\sigma(x) +\int_{\Gamma_{\xi,\epsilon}^-} (\pd_\nu u) \overline{u_{\zeta_1}} d\sigma(x).
\ee
Further, we infer from \eqref{f3b}, the estimate \eqref{e9} with $s=1$, and the continuity of the trace from $H^1(\check{\Omega})$ into $L^2(\check{\Gamma})$, that
\bea
\abs{ \int_{\Gamma_{\xi,\epsilon}^{\pm}} (\pd_\nu u) \overline{u_{\zeta_1}} d\sigma(x)} 
& \leq & \int_{\Gamma_{\xi,\epsilon}^{\pm}} \abs{(\pd_\nu u ) e^{-\tau \xi \cdot x'} (1+v_{\zeta_1}(x))} d\sigma(x') dx_1 \nonumber \\
& \leq & C \| e^{-\tau\xi\cdot x'} \pd_\nu u \|_{L^2(\Gamma_{\xi,\epsilon}^{\pm})}, \label{f8}
\eea
where $C$ is another positive constant depending only on $\omega$ and $M_+$. Moreover, we have
$$ 
\epsilon \| e^{-\tau \xi \cdot x'} \pd_\nu u \|_{L^2(\Gamma_{\xi,\epsilon}^+)}^2 \leq \| e^{-\tau \xi \cdot x'} (\xi \cdot \nu')^{\frac{1}{2}} \pd_\nu u \|_{L^2(\Gamma_{\xi,\epsilon}^+)}^2
\leq \| e^{-\tau \xi \cdot x'} (\xi \cdot \nu')^{\frac{1}{2}} \pd_\nu u \|_{L^2(\Gamma_{\xi}^+)}^2, $$
from the very definition of $\Gamma_{\xi,\epsilon}^+$ and the imbedding $\Gamma_{\xi,\epsilon}^+ \subset \Gamma_{\xi}^+$. Therefore, applying the Carleman estimate of Lemma \ref{lm-ce} to the solution $u$ of \eqref{f6}, which is possible since $\tau \geq \tau_0$, we get that
\bea
\tau \epsilon \| e^{-\tau \xi \cdot x'} \pd_\nu u \|_{L^2(\Gamma_{\xi,\epsilon}^+)}^2 
& \leq & \| e^{-\tau \xi \cdot x'} (-\Delta+V_1) u \|_{L^2(\check{\Omega})}^2 + \tau \| e^{-\tau \xi \cdot x'} |\xi \cdot \nu'|^{\frac{1}{2}} \pd_\nu u \|_{L^2(\Gamma_\xi^-)}^2 \nonumber \\
& \leq & \| e^{-\tau \xi \cdot x'} V u_{\zeta_2} \|_{L^2(\check{\Omega})}^2 + \tau \| e^{-\tau \xi \cdot x'} |\xi \cdot \nu'|^{\frac{1}{2}} \pd_\nu u \|_{L^2(\Gamma_\xi^-)}^2. \label{f9}
\eea
Here we used the fact, arising from the density of $\{ u \in \cC_\theta^2 \left([0,1] \times \overline{\omega} \right),\ u_{\vert \check{\Gamma}}=0 \}$ in $\{ u \in \cH_\theta^2 (\check{\Omega} ),\ u_{\vert \check{\Gamma}}=0 \}$, that the Carleman estimate \eqref{ce} is still valid for $u \in \cH_\theta^2 (\check{\Omega} )$ obeying $u_{\vert \check{\Gamma}}=0$.

Next, as 
$$e^{-\tau \xi \cdot x'} V u_{\zeta_2}(x) =  e^{-\tau \xi \cdot x'} V  e^{\zeta_2 \cdot x} (1+v_{\zeta_2}(x)) =  e^{-i \left( \pi k x_1 + \frac{\eta \cdot x'}{2} - \ell \cdot x \right)} V (1+v_{\zeta_2}(x)),\ x \in \check{\Omega}, $$
%$$  e^{-\tau \xi \cdot x'} V u_{\zeta_2}(x) =  e^{-\tau \xi \cdot x'} V  e^{\zeta_2 \cdot x} (1+v_{\zeta_2}(x)) =  e^{-i \left( \frac{\pi k}{2} x_1 + \frac{\eta \cdot x'}{2} - \ell \cdot x \right)} V %(1+v_{\zeta_2}(x)), $$
by \eqref{e7} and \eqref{f4b}, we have $| e^{-\tau \xi \cdot x'} V u_{\zeta_2}(x) | = | V(x) | | 1 + v_{\zeta_2}(x) |$, according to \eqref{e4}, so it follows from \eqref{e13} that
\bel{f10}
\| e^{-\tau \xi \cdot x'} V u_{\zeta_2} \|_{L^2(\check{\Omega})} \leq M_+ \left( | \omega | + \frac{C}{\tau} \right),
\ee
with $C=C(\omega,M_+,F')>0$.
Further, bearing in mind  that $\pd \omega_{\xi}^- \subset \pd \omega_{\xi,\epsilon}^-$ and $| \xi \cdot \nu' | \leq 1$ on $\pd \omega_{\xi,\epsilon}^-$, we get that
$$ \| e^{-\tau \xi \cdot x'} |\xi \cdot \nu' |^{\frac{1}{2}} \pd_\nu u \|_{L^2(\Gamma_\xi^-)} \leq \| e^{-\tau \xi \cdot x'} \pd_\nu u \|_{L^2(\Gamma_{\xi,\epsilon}^-)}. $$
This and \eqref{f9}-\eqref{f10} yield the estimate
$$
\| e^{-\tau \xi \cdot x'} \pd_\nu u \|_{L^2(\Gamma_{\xi,\epsilon}^+)}^2 
\leq \frac{C}{\epsilon} \left( \frac{1}{\tau} + \| e^{-\tau \xi \cdot x'} \pd_\nu u \|_{L^2(\Gamma_{\xi,\epsilon}^-)}^2 \right),
$$
for some constant $C=C(\omega, M_+, F')>0$.
We infer from this and \eqref{f7}-\eqref{f8}, that
\bel{f11}
\abs{\int_{\check{\Omega}} V u_{\zeta_2} \overline{u_{\zeta_1}} dx} \leq C \left( \frac{1}{\tau} + \| e^{-\tau \xi \cdot x'} \pd_\nu u \|_{L^2(\Gamma_{\xi,\epsilon}^-)}^2
\right)^{\frac{1}{2}},
\ee
where $C$ is a positive constant $C$ depending on $\omega$, $M_+$, $F'$, and $G'$. 

On the other hand, with reference to \eqref{e7}-\eqref{e8}, \eqref{f3b}, and \eqref{f4b}, we find through direct calculation that
\bel{f12}
\int_{\check{\Omega}} V u_{\zeta_2} \overline{u_{\zeta_1}} dx= \int_{(0,1) \times \R^2} e^{-i ( 2 \pi k x_1 + \eta \cdot x')} V(x_1,x') dx_1dx'+ \int_{\check{\Omega}} R(x)  dx,
\ee
where
$$ R(x) := e^{-i( 2 \pi k x_1 + \eta \cdot x')} V(x) \left( v_{\zeta_2}(x)+ \overline{v_{\zeta_1}(x)} +v_{\zeta_2}(x) \overline{v_{\zeta_1}(x)} \right),\ x=(x_1,x') \in \check{\Omega}.$$ 
Since $\| v_{\zeta_1} \|_{L^2(\check{\Omega})}$ (resp., $\| v_{\zeta_2} \|_{L^2(\check{\Omega})}$) is bounded, up to some multiplicative constant, from above by $\tau^{-1}$ (resp., $\tau^{-\frac{1}{2}}$), according to \eqref{e9} (resp., \eqref{e13}), we obtain that
$$ \abs{\int_{\check{\Omega}} R(x)  dx} \leq M_+ \left( | \omega |^{\frac{1}{2}} \left( \| v_{\zeta_1}\|_{L^2(\check{\Omega})} + \| v_{\zeta_2} \|_{L^2(\check{\Omega})} \right) + 
\| v_{\zeta_1} \|_{L^2(\check{\Omega})}\| v_{\zeta_2} \|_{L^2(\check{\Omega})} \right) \leq C \tau^{-\frac{1}{2}}, $$
where $C$ is independent of $\tau$. It follows from this and \eqref{f11}-\eqref{f12} that
\bea
\abs{\int_{(0,1) \times \R^2} e^{-i( 2 \pi k x_1 + \eta\cdot x')} V(x) dx}^2 
%\leq \frac{C}{\tau}+ 
& \leq & C \left( \frac{1}{\tau}+ \| e^{-\tau \xi \cdot x'} \pd_\nu u \|_{L^2(\Gamma_{\xi,\epsilon}^-)}^2 \right) \nonumber \\
& \leq & C \left( \frac{1}{\tau} + e^{c_\omega \tau} \norm{\pd_\nu u}_{L^2(\Gamma_{\xi,\epsilon}^-)}^2 \right), \label{f13}
\eea
where we recall that $c_\omega:=\max \{ |x'|; x' \in \overline{\omega} \}$, and $C=C(\omega,M_\pm,F',G')>0$. Finally, upon recalling that $u=w_1-u_{\zeta_2}$, where $w_1$ is solution to \eqref{f5} and $u_{\zeta_2}$ satisfies \eqref{eq4} with $V=V_2$, we see that
$$ \pd_\nu u=(\Lambda_{{V}_2,\theta} - \Lambda_{{V}_1,\theta} ) f,\ f=\mathcal T_{0,\theta} u_{\zeta_2}. $$
Since $\pd \omega_{\xi,\epsilon}^- \subset G'$, by \eqref{f1}, we have $\norm{\pd _\nu u}_{L^2(\Gamma_{\xi,\epsilon}^-)} \leq  \norm{\Lambda_{{V}_2,\theta}-\Lambda_{{V}_1,\theta}} \norm{\mathcal T_{0,\theta}u_{\zeta_2}}_{\sH_\theta(\check{\Gamma})}$, and hence
$$
\norm{\pd _\nu u}_{L^2(\Gamma_{\xi,\epsilon}^-)} \leq C \norm{\Lambda_{{V}_2,\theta}-\Lambda_{{V}_1,\theta}} \| u_{\zeta_2} \|_{H_\Delta(\check{\Omega})}
\leq C e^{c_\omega \tau} \norm{\Lambda_{{V}_2,\theta}-\Lambda_{{V}_1,\theta}},
$$
by \eqref{e16}, where $C=C(\omega,M_+,F')>0$. This and \eqref{f13} entail \eqref{f0}.
\end{proof}

Now the result of Theorem \ref{thm2} follows from Lemma \ref{lm-pre} by arguing as in the derivation (see \cite[Section 6]{CKS2}) of \cite[Theorem 3.3]{CKS2} from \cite[Lemma 6.1]{CKS2}.

\section{Proof of Corollary \ref{cor-a}}
\label{sec-proofcor}
It is well known that if $u$ is a solution to \eqref{a-eq1}, then the Liouville transform of $u$, $v:=a^{\frac{1}{2}} u$, solves the following BVP
$$
\left\{
\begin{array}{rcll} 
(-\Delta + V_a ) v & = & 0, & \mbox{in}\ \Omega,\\ 
v & = & a^{\frac{1}{2}} f, & \mbox{on}\ \Gamma,
\end{array}
\right.
$$
where we recall that $V_a:=a^{-\frac{1}{2}} \Delta a^{\frac{1}{2}}$. Moreover, standard computations yield that
$$
\Sigma_a f=a^{\frac{1}{2}} \Lambda_{V_a} a^{\frac{1}{2}}f - a^{\frac{1}{2}} \left( \pd_\nu a^{\frac{1}{2}} \right) f,\ f \in \sK(\Gamma) \cap a_1^{-\frac{1}{2}}(\sH_c(F)),
$$
where $\Lambda_{V_a}$ are $\Sigma_a$ are defined by \eqref{es0} and \eqref{ca2}, respectively.
From this and \eqref{ca3}-\eqref{ca4}, it then follows for every $f \in \sK(\Gamma) \cap a_1^{-\frac{1}{2}}(\sH_c(F))$, that
$$
\Sigma_{a_j} f = a_1^{\frac{1}{2}}\Lambda_{V_j} a_1^{\frac{1}{2}} f - a_1^{\frac{1}{2}} \left( \pd_\nu a_1^{\frac{1}{2}} \right) f_{\vert G},\  j=1,2,
$$
where, for simplicity, $V_j$ stands for $V_{a_j}$. As a consequence we have
\bel{g2}
(\Sigma_{a_1}-\Sigma_{a_2}) f =a_1^{\frac{1}{2}}(\Lambda_{V_1}-\Lambda_{V_2}) a_1^{\frac{1}{2}}f,\  f \in \sK(\Gamma) \cap a_1^{-\frac{1}{2}}(\sH_c(F)).
\ee
Since $a_j \in \sA_\omega(c_*,M_\pm)$, $j=1,2$, it is clear that $V_j \in \sV_\omega(M_\pm)$, whence
$\Lambda_{V_1}-\Lambda_{V_2} \in \cB(\sH_c(F),L^2(G))$, by \eqref{es1}. Thus, $(\Sigma_{a_1}-\Sigma_{a_2})f \in L^2(G)$, by \eqref{g2}, and we have the the estimate
$$
\norm{(\Sigma_{a_1}-\Sigma_{a_2})f}_{L^2(G)} \leq \| a_1 \|_{L^{\infty}(\Gamma)}^{\frac{1}{2}} \norm{\Lambda_{V_1}-\Lambda_{V_2}}_{\cB(\sH_c(F),L^2(G))} \norm{a_1^{\frac{1}{2}}f}_{\sH(\Gamma)},\ f \in \sK(\Gamma) \cap a_1^{-\frac{1}{2}}(\sH_c(F)).
$$
Since the space $a_1^{-\frac{1}{2}} \sH_c(F)$ is endowed with the norm $\| a_1^{\frac{1}{2}} \cdot \|_{\sH(\Gamma)}$, this entails that
$\Sigma_{a_1}-\Sigma_{a_2}$ can be extended to a bounded operator from $a_1^{-\frac{1}{2}} \sH_c(F)$ into $L^2(G)$, fulfilling
$$
\| \Sigma_{a_1}-\Sigma_{a_2} \| \leq \| a_1 \|_{L^{\infty}(\Gamma)}^{\frac{1}{2}} \norm{\Lambda_{V_1}-\Lambda_{V_2}}_{\cB(\sH_c(F),L^2(G))}.
$$
Moreover, we notice from \eqref{ca1} and \eqref{g2} that
\bel{g4}
\norm{\Lambda_{V_1}-\Lambda_{V_2}}_{\cB(\sH_c(F),L^2(G))} \leq a_*^{-\frac{1}{2}} \norm{\Sigma_{a_1}-\Sigma_{a_2}}.
\ee
Having established \eqref{g4}, we turn now to proving the stability estimate \eqref{ca5}.
To this purpose, we put $\alpha:=a_1^{\frac{1}{2}}-a_2^{\frac{1}{2}}$, get through basic computations that
$$
\Delta \alpha=  \Delta a_1^{\frac{1}{2}} - \Delta a_2^{\frac{1}{2}} = a_1^{\frac{1}{2}} V_1 - a_2^{\frac{1}{2}} V_2=\alpha V_1 + a_2^{\frac{1}{2}}(V_1-V_2),
$$
and then deduce from \eqref{a-per} and \eqref{ca3}-\eqref{ca4}, that
\bel{g5}
\left\{
\begin{array}{rcll} 
(-\Delta + V_1 ) \alpha & = & -a_2^{\frac{1}{2}}(V_1-V_2), & \mbox{in}\ \check{\Omega},\\ 
\alpha(1,\cdot) - \alpha(0,\cdot) & = & 0, & \mbox{in}\ \omega, \\
\pd_{x_1} \alpha(1,\cdot) - \pd_{x_1} \alpha(0,\cdot) & = & 0, & \mbox{in}\ \omega, \\
\alpha & = & 0 , & \mbox{on}\ \check{\Gamma}.
\end{array}
\right.
\ee
Since $a_1 \in \sA_\omega(c_*,M_\pm)$, then \eqref{g5} admits a unique solution 
%$\alpha \in H^1_{0,\mathrm{per}}(\check{\Omega})$, by the standard variational theory, where
$$ \alpha \in H^1_{0,\mathrm{per}}(\check{\Omega}) :=\{ v\in H^1(\check{\Omega});\ v_{|\check{\Gamma}}=0\ \mbox{and}\ v(1,\cdot) - v(0,\cdot) = 0\ \mbox{in}\ \omega \}, $$
by the standard variational theory.
%Taking the $L^2(\check{\Omega})$-scalar product of $\alpha$ with the first line of \eqref{g5}, we get upon integrating by parts, that
Moreover, there exists a constant $C=C(\omega,a_*,M_\pm)>0$, such that we have
\bel{g6} 
\norm{\alpha}_{H^1(\check{\Omega})} \leq C \norm{a_2^{\frac{1}{2}}(V_2-V_1)}_{(H^1_{0,\mathrm{per}}(\check{\Omega}))'} \leq C \norm{V_1-V_2}_{(H^1_{0,\mathrm{per}}(\check{\Omega}))'}.
\ee
Here, $(H^1_{0,\mathrm{per}}(\check{\Omega}))'$ denotes the dual space to $H^1_{0,\mathrm{per}}(\check{\Omega})$. 

Next, arguing as in the derivation of Theorem \ref{thm1}, we find that
$$ \norm{V_1-V_2}_{(H^1_{0,\mathrm{per}}(\check{\Omega}))'}\leq C \Phi \left( \norm{\Lambda_{V_1}-\Lambda_{V_2}}_{\cB(\sH_c(F),L^2(G))} \right), $$
%$$ \norm{V_1-V_2}_{H^{-1}(\check{\Omega})} \leq C \Phi \left( \norm{\Lambda_{V_1}-\Lambda_{V_2}}_{\cB(\sH_c(F),L^2(G))} \right), $$
so we deduce from \eqref{g4} and \eqref{g6} that
$$ \norm{a_1^{\frac{1}{2}}-a_2^{\frac{1}{2}}}_{H^1(\check{\Omega})} \leq  C \Phi \left(a_*^{-\frac{1}{2}}\norm{\Sigma_{a_1}-\Sigma_{a_2}}\right). $$
Here we used the fact that $\Phi$ is a non-decreasing function on $[0,+\infty)$.
Now, the desired result follows from this upon noticing that
\beas
\norm{a_1-a_2}_{H^1(\check{\Omega})} & = & \norm{(a_1^{\frac{1}{2}}+a_2^{\frac{1}{2}})(a_1^{\frac{1}{2}}-a_2^{\frac{1}{2}})}_{H^1(\check{\Omega})} \\
& \leq & 
a_*^{-\frac{1}{2}} \left( \norm{a_1}_{W^{1,\infty}(\check{\Omega})} + \norm{a_2}_{W^{1,\infty}(\check{\Omega})} \right) \norm{a_1^{\frac{1}{2}}-a_2^{\frac{1}{2}}}_{H^1(\check{\Omega})} \\
& \leq & 2 a_*^{-\frac{1}{2}} M_+ \norm{a_1^{\frac{1}{2}}-a_2^{\frac{1}{2}}}_{H^1(\check{\Omega})}.
\eeas

\appendix

\section{Characterizing the space $H_{\Delta,\theta}(\check{\Omega})$}
\label{sec-app}

In this appendix we establish that any function $v \in H_{\Delta}(\check{\Omega})$ satisfying the condition \eqref{i7} for some fixed $\theta \in [0, 2 \pi)$, actually belongs to $H_{\Delta}(\check{\Omega})$. 

%In this appendix we characterize the space $H_{\Delta,\theta}(\check{\Omega})$, for $\theta \in [0, 2 \pi)$ fixed, as the set of those functions $v \in H_{\Delta}(\check{\Omega})$, %satisfying the condition \eqref{i7}.

As a warm up, we fix $\theta \in [0,2\pi)$ and notice for each $v \in H_{\Delta}(\check{\Omega})$, that
\bel{c1} 
\left( v \in H_{\Delta,\theta}(\check{\Omega}) \right) \Longleftrightarrow \left( \langle \Delta v , w \rangle_{L^2(\check{\Omega})} = \langle v , \Delta w \rangle_{L^2(\check{\Omega})},\
w \in H_\theta^2(0,1;H^2_0(\omega)) \right),
\ee
where $H^2_0(\omega)$ is the closure of $C^\infty_0(\omega)$ for the topology of $H^2(\omega)$.
This can be seen from the the following identity
\beas
& & \langle \Delta v , w \rangle_{L^2(\check{\Omega})} - \langle v , \Delta w \rangle_{L^2(\check{\Omega})} \\
& = & 
\langle \pd_{x_1} v(1,\cdot) , w(1,\cdot) \rangle_{H^{-2}(\omega),H_0^2(\omega)} -
\langle \pd_{x_1} v(0,\cdot) , w(0,\cdot) \rangle_{H^{-2}(\omega),H_0^2(\omega)} \\
& & - \left( \langle u(1,\cdot) , \pd_{x_1} w(1,\cdot) \rangle_{H^{-2}(\omega),H_0^2(\omega)} -
\langle v(0,\cdot) , \pd_{x_1} w(0,\cdot) \rangle_{H^{-2}(\omega),H_0^2(\omega)} \right),
\eeas
which holds true for any $v \in H_{\Delta}(\check{\Omega})$ and $w \in H^2(0,1;H^2_0(\omega))$, and leads to
\beas
& & \langle \Delta v , w \rangle_{L^2(\check{\Omega})} - \langle v , \Delta w \rangle_{L^2(\check{\Omega})} \\
%& = & 
%\langle \pd_{x_1} v(1,\cdot) , e^{i \theta} w(0,\cdot) \rangle_{H^{-2}(\omega),H_0^2(\omega)} 
%- \langle \pd_{x_1} v(0,\cdot) , w(0,\cdot) \rangle_{H^{-2}(\omega),H_0^2(\omega)} \\
%& & - \left( \langle v(1,\cdot) , e^{i \theta} \pd_{x_1} w(0,\cdot) \rangle_{H^{-2}(\omega),H_0^2(\omega)} -
%\langle v(0,\cdot) , \pd_{x_1} w(0,\cdot) \rangle_{H^{-2}(\omega),H_0^2(\omega)} \right) \\
& = & \langle e^{-i \theta} \pd_{x_1} v(1,\cdot) - \pd_{x_1} v(0,\cdot) , w(0,\cdot) \rangle_{H^{-2}(\omega),H_0^2(\omega)} 
- \langle  e^{-i \theta} v(1,\cdot) - v(0,\cdot), \pd_{x_1} w(0,\cdot) \rangle_{H^{-2}(\omega),H_0^2(\omega)},
\eeas
as soon as $w \in H_\theta^2(0,1;H^2_0(\omega))$. Therefore, the right hand side of \eqref{c1} may be equivalently reformulated as
\bea
& & \langle \pd_{x_1} v(1,\cdot) - e^{i \theta} \pd_{x_1} v(0,\cdot) , w(0,\cdot) \rangle_{H^{-2}(\omega),H_0^2(\omega)} \nonumber \\
& - & \langle v(1,\cdot) - e^{i \theta} v(0,\cdot), \pd_{x_1} w(0,\cdot) \rangle_{H^{-2}(\omega),H_0^2(\omega)} =0,\ w \in H_\theta^2(0,1;H^2_0(\omega)). \label{c2}
\eea
Taking $w(x)=e^{i \theta x_1} \sin(2 \pi x_1) \psi(x')$, where $\psi$ is arbitrary in $H^2_0(\omega)$, in \eqref{c2}, we get that
$$
\langle v(1,\cdot) - e^{i \theta} v(0,\cdot) , \psi \rangle_{H^{-2}(\omega),H_0^2(\omega)} =0,
$$
and hence that $v(1,\cdot) - e^{i \theta} v(0,\cdot)=0$. Doing the same with $w(x)=e^{i \theta x_1} \psi(x')$, we find that
$$
\langle \pd_{x_1} v(1,\cdot) - e^{i \theta} \pd_{x_1} v(0,\cdot) , \psi \rangle_{H^{-2}(\omega),H_0^2(\omega)} =0,
$$
which yields $\pd_{x_1} v(1,\cdot) - e^{i \theta} \pd_{x_1} v(0,\cdot)=0$, and proves \eqref{c1}.

Armed with \eqref{c1} we may now establish the main result of this appendix.

\begin{lemma}
\label{lm-a}
Fix $\theta \in [0,2\pi)$, and let $v \in H_{\Delta}(\check{\Omega})$ satisfy the condition \eqref{i7}. Then $v$ belongs to $H_{\Delta,\theta}(\check{\Omega})$.
\end{lemma}

\begin{proof}
Since $v \in H_{\Delta}(\check{\Omega})$ satisfies the condition \eqref{i7}, then both identities
\bel{c4}
v = \sum_{k \in \Z} \hat{v}_{k,\theta} \varphi_{k,\theta}\ \mbox{and}\ \Delta v= \sum_{k \in \Z} \left( \Delta'-(\theta + 2 k \pi)^2 \right) \hat{v}_{k,\theta} \varphi_{k,\theta},
\ee
hold in $L^2(\check{\Omega})$, where $\hat{v}_{k,\theta}$ and $\varphi_{k,\theta}$ are defined by \eqref{i8} and \eqref{i9}, respectively. Moreover, as $L^2(\check{\Omega})=L^2(0,1;L^2(\omega))$ and $\{ \varphi_{k,\theta},\ k \in \Z \}$ is an orthonormal basis of $L^2(0,1)$, it follows from \eqref{c4} that
\bel{c5}
\sum_{k \in \Z} \| \hat{v}_{k,\theta} \|_{L^2(\omega)}^2 = \| v \|_{L^2(\check{\Omega})}^2\ \mbox{and}\ \sum_{k \in \Z} \| \left( \Delta'-(\theta + 2 k \pi)^2 \right) \hat{v}_{k,\theta} \|_{L^2(\omega)}^2 = \| \Delta v \|_{L^2(\check{\Omega})}^2.
\ee
Similarly, for $w \in H^2_\theta(0,1;H^2_0(\omega))$, it holds true that
\bel{c6}
w = \sum_{k \in \Z} \hat{w}_{k,\theta} \varphi_{k,\theta}\ \mbox{and}\ \Delta w= \sum_{k \in \Z} \left( \Delta'-(\theta + 2 k \pi)^2 \right) \hat{w}_{k,\theta} \varphi_{k,\theta},
\ee
in $L^2(\check{\Omega})$, and we have the energy estimates $\sum_{k \in \Z} \| \hat{w}_{k,\theta} \|_{H^2(\omega)}^2 = \| w \|_{L^2(0,1;H^2(\omega))}^2$, $\sum_{k \in \Z} (\theta + 2 k \pi)^4 \| \hat{w}_{k,\theta} \|_{L^2(\omega)}^2 = \| \pd_{x_1}^2 w \|_{L^2(\check{\Omega})}^2$, and 
$\sum_{k \in \Z} \| \Delta' \hat{w}_{k,\theta} \|_{L^2(\omega)}^2 = \| \Delta' w \|_{L^2(\check{\Omega})}^2$.
It follows from this and \eqref{c4}--\eqref{c6}, that
\beas
\langle \Delta v , w \rangle_{L^2(\check{\Omega})}
& = & \sum_{k \in \Z} \left\langle \left( \Delta'- ( \theta + 2 k \pi )^2 \right) \hat{v}_{k,\theta} , \hat{w}_{k,\theta} \right\rangle_{L^2(\omega)} \\
& = & \sum_{k \in \Z} \langle \Delta' \hat{v}_{k,\theta} , \hat{w}_{k,\theta} \rangle_{H^{-2}(\omega),H^2_0(\omega)} - \sum_{k \in \Z} \left\langle ( \theta + 2 k \pi)^2 \hat{v}_{k,\theta} , \hat{w}_{k,\theta} \right\rangle_{H^{-2}(\omega),H^2_0(\omega)} \\
& = & \sum_{k \in \Z} \langle \hat{v}_{k,\theta} , \Delta' \hat{w}_{k,\theta} \rangle_{L^2(\omega)} - \sum_{k \in \Z} \left\langle \hat{v}_{k,\theta} ,  ( \theta + 2 k \pi)^2 \hat{w}_{k,\theta} \right\rangle_{L^2(\omega)} \\
& = & \sum_{k \in \Z} \left\langle \hat{v}_{k,\theta} ,  \left( \Delta' - ( \theta + 2 k \pi)^2 \right) \hat{w}_{k,\theta} \right\rangle_{L^2(\omega)} \\
& = & \langle v , \Delta w \rangle_{L^2(\check{\Omega})}.
\eeas
This and \eqref{c1} yield the desired result.
\end{proof}

\begin{remark}
The assumption $v \in H_{\Delta}(\check{\Omega})$ in Lemma \ref{lm-a} can be weakened to $v \in L^2(\check{\Omega})$, as an $L^2(\check{\Omega})$-function satisfying condition \eqref{i7} is automatically in $H_{\Delta}(\check{\Omega})$. Moreover, it is not hard to see that the result of Lemma \ref{lm-a} can be improved significantly in order to provide the following characterization of the space $H_{\Delta,\theta}(\check{\Omega})$:
$$ H_{\Delta,\theta}(\check{\Omega}) = \left\{ v \in L^2(\check{\Omega}),\ v\ \mbox{satisfies}\ \eqref{i7} \right\}. $$
\end{remark}

\bigskip

\vspace*{.5cm}
\noindent {\sc Mourad Choulli}, Universit\'e de Lorraine, Institut Elie Cartan de Lorraine, CNRS, UMR 7502, Boulevard des Aiguillettes, BP 70239, 54506 Vandoeuvre les Nancy cedex - Ile du Saulcy, 57045 Metz cedex 01, France.\\
E-mail: {\tt  mourad.choulli@univ-lorraine.fr}. \vspace*{.1cm} \\

\noindent {\sc Yavar Kian}, Aix-Marseille Universit\'e, CNRS, CPT UMR 7332, 13288 Marseille, and Universit\'e de Toulon, CNRS, CPT UMR 7332, 83957 La Garde, France.\\
E-mail: {\tt yavar.kian@univ-amu.fr}. \vspace*{.1cm} \\

\noindent {\sc Eric Soccorsi}, Aix-Marseille Universit\'e, CNRS, CPT UMR 7332, 13288 Marseille, and Universit\'e de Toulon, CNRS, CPT UMR 7332, 83957 La Garde, France.\\
E-mail: {\tt eric.soccorsi@univ-amu.fr}.

\end{document}